\documentclass[reqno,12pt]{amsart}

\usepackage{amssymb}

\numberwithin{equation}{section}

\newcommand{\1}{\mathbf {1}}

\newcommand{\h}{{\mathfrak h}}

\newcommand{\mfS}{{\mathfrak S}}

\newcommand{\id}{\mathrm{id}}

\newcommand{\la}{\langle}
\newcommand{\ra}{\rangle}

\DeclareMathOperator{\Aut}{Aut}

\newtheorem{theorem}{Theorem}[section]
\newtheorem{proposition}[theorem]{Proposition}
\newtheorem{lemma}[theorem]{Lemma}
\newtheorem{corollary}[theorem]{Corollary}
\theoremstyle{definition}

\theoremstyle{remark}
\newtheorem{remark}[theorem]{Remark}

\newcommand{\tr}{\mathrm{tr}}

\newcommand{\w}{\omega}
\DeclareMathOperator{\Irr}{Irr\,}
\DeclareMathOperator{\qdim}{{\rm qdim}}
\newcommand{\fusion}[3]{{\binom{#3}{#1\;#2}}}
\newcommand{\huto}[1]{\mbox{\boldmath $#1$}}
\newcommand\R{\mathbb{R}}
\newcommand\Z{\mathbb{Z}}
\newcommand\Zpos{\Z_{\geq0}}
\newcommand\Zplus{\Z_{>0}}
\newcommand\Q{\mathbb{Q}}

\newcommand\C{\mathbb{C}}
\newcommand\N{\mathbb{N}}

\newcommand{\NO}{\,{\raise0.25em\hbox{$\mathop{\hphantom {\cdot}}\limits^{_{\circ}}_{^{\circ}}$}}\,}

\begin{document}

\title[Extensions of cyclic orbifold $V_{\sqrt{2}A_{p-1}^d}^{\langle \widehat{\sigma} \rangle}$]
{Extensions of tensor products of $\Z_p$-orbifold models of the lattice vertex operator 
algebra $V_{\sqrt{2}A_{p-1}}$}

\author[T. Abe]{Toshiyuki Abe}
\address{Faculty of Education, 
Ehime University, Matsuyama, Ehime 790-8577, Japan}
\email{abe.toshiyuki.mz@ehime-u.ac.jp}

\author[C.H. Lam]{Ching Hung Lam}
\address{Institute of Mathematics, Academia Sinica, Taipei 115, Taiwan}
\email{chlam@math.sinica.edu.tw}

\author[H. Yamada]{Hiromichi Yamada}
\address{Department of Mathematics, Hitotsubashi University, Kunitachi,
Tokyo 186-8601, Japan, 
Institute of Mathematics, Academia Sinica, Taipei 115, Taiwan}
\email{yamada.h@r.hit-u.ac.jp}

\thanks{T.A. is partially supported by JSPS fellow 15K04823. 
C.L. is partially supported by MoST grant 104-2115-M-001-004-MY3 of Taiwan. }



\begin{abstract}
Let $p$ be an odd prime and let $\widehat{\sigma}$ be an order $p$ automorphism of 
$V_{\sqrt{2}A_{p-1}}$ which is a lift of a $p$-cycle in the Weyl group 
${\rm Weyl}(A_{p-1})\cong \mfS_p$. 
We study a certain extension $V$ of a tensor product of finitely many copies of 
the orbifold model $V_{\sqrt{2}A_{p-1}}^{\la \widehat{\sigma} \ra}$ and give a criterion for $V$ 
that every irreducible $V$-module is a simple current.  
\end{abstract}

\maketitle

\section{Introduction}
For any positive definite even lattice $L$, one can associate a rational, $C_2$-cofinite  
vertex operator algebra (VOA)  $V_L$ of CFT type. 
The representation theory of $V_L$, including a theory of twisted modules, has been well  
developed (see \cite{BakalovKac04}, \cite{Dong93}, \cite{DongLepowsky93}, \cite{DongLepowsky96}, 
\cite{FLM}, \cite{Lepowsky85}, or \cite{LepowskyLi04} and references therein).    
An isometry $\sigma$ of $L$ can be lifted to an automorphism $\widehat{\sigma}$ of 
$V_L$ of the same order or its twice. 
A construction of irreducible $\widehat{\sigma}$-twisted $V_L$-modules is obtained 
in \cite{Lepowsky85} (see \cite{BakalovKac04}, \cite{DongLepowsky96} also).  

In this paper, we calculate the quantum dimensions of irreducible  
$\widehat{\sigma}$-twisted $V_L$-modules when $\sigma$ is fixed-point-free on $L$, 
that is, $\sigma$ has no nontrivial fixed point in $L$.  
Furthermore, we apply the result to the case $L$ is a certain lattice $M$ and the order of 
$\sigma$ is an odd prime. 
We give a necessary and sufficient condition on the lattice $M$ for which 
every irreducible module of the orbifold model $V_{M}^{\la \widehat{\sigma} \ra}$ 
is a simple current. 
There is a more general result by S. M\"{o}ller \cite{Moller16}, where the modular transformations of 
characters are studied. 
Our settings and arguments in this paper are slightly different from \cite{Moller16}. 

We explain the construction of $M$ in detail. 
Let $N=\sqrt{2}A_{p-1}$ be $\sqrt{2}$ times a root lattice of type $A_{p-1}$. 
Then $N^\circ/N \cong \Z_2^{p-2} \times \Z_{2p}$, where $N^\circ$ is the dual lattice of $N$. 
Assume that $p\geq 3$ is an odd integer. 
Then we have $N^\circ/N \cong \Z_2^{p-1} \times \Z_{p}$. 
The inner product of $N^\circ$   
and the minimum values of square norms of elements in the cosets 
of $N$ in $N^\circ$ give rise to a code structure on $N^\circ/N\cong \Z_2^{p-1} \times\Z_{p}$, 
that is, an inner product and a weight function on $\Z_2^{p-1} \times\Z_{p}$.  
Let $k=\Z_2^{p-1}$ and $l=\Z_p$ be the corresponding subcodes in $N^\circ/N$, 
so that $N^\circ/N = k \times l$. 

For $d\in\Zplus$, let $N^d$ (resp. $(N^\circ)^d$) be a direct sum of $d$ copies of 
$N$ (resp. $N^\circ$). 
Then we can extend the code structure on $N^\circ/N = k\times l$ to 
$(N^\circ)^d/N^d = (N^\circ/N)^d = k^d\times l^d$ naturally. 
There is a one to one   
correspondence between $\Z$-submodules $\mathcal{E}$ of $k^d\times l^d$ 
and sublattices $L_{\mathcal{E}}$ of $(N^\circ)^d$ containing $N^d$. 
Some properties of the lattice $L_{\mathcal{E}}$ will be discussed in Section \ref{Sect4}. 
It turns out that the lattice $L_\mathcal{E}$ is  even, integral and unimodular if and only if the 
corresponding $\Z$-submodules $\mathcal{E}$ is even, self-orthogonal and self-dual, 
respectively.  

Let $\sigma$ be an isometry of $N$ of order $p$ which corresponds to a Coxeter element 
of the Weyl group ${\rm Weyl}(A_{p-1})\cong \mfS_p$. 
Then $\sigma$ is fixed-point-free on $N$. 
Moreover, $\sigma$ induces an isometry of $N^\circ$ and acts naturally on $N^\circ/N = k \times l$. 
In fact, $\sigma$ is fixed-point-free on $k$; on the other hand, all elements in $l$ are fixed by 
$\sigma$.   
The isometry $\sigma$ can be extended to a fixed-point-free isometry of $N^d$ of order $p$ 
by diagonal action on each direct summand. 
Then $\sigma$ induces an automorphism of $k^d\times l^d$, which is 
fixed-point-free on $k^d$ and trivial on $l^d$. 
If $\mathcal{E}\subset k^d\times l^d$ is $\sigma$-invariant, 
then $\sigma$ acts on $L_{\mathcal{E}}$ as a fixed-point-free isometry. 

We consider the case $\mathcal{E}=\mathcal{C} \times \mathcal{D}$ for a $\sigma$-invariant 
even $\Z$-submodule $\mathcal{C}\subset k^d$ and 
an even $\Z$-submodule $\mathcal{D}\subset l^d$.   
Set $M=L_{\mathcal{C}\times \mathcal{D}}$.
Then $M$ becomes an even lattice and $\widehat{\sigma}$ defines an automorphism of order $p$ 
of the VOA $V_M$. 
For any $i=1, \ldots, p-1$, we can construct irreducible $\widehat{\sigma}^i$-twisted 
$V_{M}$-modules following \cite{Lepowsky85}.
As one of the main theorems, we show that the quantum dimension of irreducible 
$\widehat{\sigma}^i$-twisted $V_{M}$-module is equal to $|\mathcal{C}^\perp/\mathcal{C}|$. 
This implies that $V_{M}^{\langle \widehat{\sigma} \rangle}$ has a group-like fusion 
if and only if $\mathcal{C}$ is self-dual. 

One of the motivation of this paper lies in a study of the Moonshine VOA $V^\natural$ 
whose full automorphism group is the Monster simple group. 
The VOA $V^\natural$ was constructed as a $\Z_2$-orbifold construction from the Leech 
lattice VOA $V_{\Lambda}$ with respect to a lift of the $-1$-isometry of $\Lambda$ \cite{FLM}. 
Since then, some different constructions have been given. 
For instance, $V^\natural$ was constructed as a framed VOA \cite{Miyamoto04}. 
There is also a $\Z_3$-orbifold construction of  $V^\natural$ from $V_{\Lambda}$  
\cite{ChenLamShimakura16}.  
Using these constructions, several $2$-local and $3$-local subgroups of the Monster simple group
have been described relatively explicitly in \cite{ChenLamShimakura16, Shimakura11}. 

The case $p = 3$ of this paper was studied in \cite{ChenLam16} 
and the representation theory of $V_{M}^{\langle \widehat{\sigma} \rangle}$ with $p = 3$ and 
$d = 12$ plays an important role in \cite{ChenLamShimakura16}. 
Recently, a $\Z_p$-orbifold construction of $V^\natural$ from 
$V_{\Lambda}$ for $p=3,5,7,13$ was obtained \cite{AbeLamYamada17}.    
We hope the results of this paper could be helpful for the study of some $p$-local subgroups 
of the Monster simple group.

This paper is organized as follows.
In Section \ref{Sect2}, we review some notions from the representation theory and 
the orbifold theory of VOAs.
In Section \ref{Sect2.1}, we recall the basics of twisted modules for VOAs. 
We recall the definition and some properties of quantum dimensions in Section \ref{Sect2.2}. 
In Sections \ref{Sect3.1}--\ref{Sect3.3}, 
we review the constructions of lattice VOAs and their twisted modules. 
We also calculate the quantum dimensions of twisted modules of lattice VOAs in Section \ref{Sect3.4}. 
In Section \ref{Sect4}, we study a lattice $N=\sqrt{2}A_{p-1}$ 
for an odd integer $p\geq 3$ and discuss the codes associated with $N$. 
In Section \ref{Sect4.1}, we 
develop a code theory for $k\times l=\Z_2^{p-1}\times \Z_p\cong N^\circ/N$.
In Section \ref{Sect4.2}, we study the action of the isometry $\sigma$ 
on the lattice $L_{\mathcal{C} \times \mathcal{D}}$.
Section \ref{Sect5} is for the orbifold model 
$V_{L_{\mathcal{C} \times \mathcal{D}}}^{\la \widehat{\sigma} \ra}$. 
In Section \ref{Sect5.1}, we consider the lattice VOA associated with the lattice 
$L_{\mathcal{C} \times \mathcal{D}}$ for a $\sigma$-invariant even $\Z$-submodule  
$\mathcal{C} \subset k^d$ and an even $\Z$-submodule $\mathcal{D} \subset l^d$.
We also give a criterion that $V_{L_{\mathcal{C} \times \mathcal{D}}}^{\la \widehat{\sigma} \ra}$ 
has a group-like fusion.  
We determine the group structure of 
$\Irr(V_{L_{\mathcal{C} \times \mathcal{D}}}^{\la \widehat{\sigma} \ra})$ when 
$V_{L_{\mathcal{C}\times \mathcal{D}}}^{\la \widehat{\sigma} \ra}$ has a group-like fusion 
in Section \ref{Sect5.2}.   

In Appendix, we give an extension of a $\sigma$-invariant alternating $\Z$-bilinear map on 
$L_{\mathcal{C} \times \mathcal{D}}$, which is used for the construction of $\sigma$-twisted 
$V_{L_{\mathcal{C} \times \mathcal{D}}}$-modules, to the case 
$L_{\mathcal{C} \times \mathcal{D}}$ is a rational lattice. 

\medskip

\paragraph{Acknowledgment:} 
The authors thank  Masahiko Miyamoto, Hiroki Shimakura and Hiroshi Yamauchi 
for useful comments and discussions.
T. A. is partially supported by JSPS fellow 15K04823. 
C. L. is partially supported by MoST grant 104-2115-M-001-004-MY3 of Taiwan.

\section{Preliminaries}\label{Sect2}

In this section, we recall 
the representation theory and the orbifold theory of VOAs 
(see \cite{CarnahanMiyamoto16}, \cite{DongJiaoXu13}, \cite{DongRenXu15}, \cite{FHL},  
\cite{LepowskyLi04} and \cite{MatsuoNagatomo99} for example).
Throughout this paper, $\xi_s$ denotes a primitive $s$-th root of unity.    
The set of all nonnegative (resp. positive) integers is denoted by $\Zpos$ (resp. $\Zplus$). 

\subsection{Twisted modules for vertex operator algebras}\label{Sect2.1}
Let $(V, Y, \1,\w) $ be a VOA.
An automorphism of $V$ is a linear map $g\in GL(V)$ such that $g(Y(a,z)b)=Y(ga,z)gb$ 
for $a,b\in V$, $g(\w)=\w$ and $g(\1)=\1$. 
The group of all automorphisms of $V$ is denoted by $\Aut(V)$. 
For a finite subgroup $G$ of $\Aut(V)$, the fixed point subspace  
\[
V^G=\{a\in V|g(a)=a\text{ for }g\in G\}
\] 
becomes a subVOA of $V$. 
The VOA $V^G$ is called an {\it orbifold model} of $V$. 

Let $T\in\Zplus$ and $g\in \Aut(V)$ satisfying $g^T=1$. 
Then 
\[
V=\bigoplus_{i=0}^{T-1}V^{(i;g)}, \quad V^{(i;g)}=\{a\in V|g(a)=\xi_T^{-i} a\} 
\] 
 (cf. \cite[Remark 3.1]{DLM2000}). 

A \textit{weak $g$-twisted $V$-module} is a vector space $M$ equipped with a linear map  
\[
Y_M: V\otimes M\rightarrow M((x^{1/T})),\quad 
a\otimes u\mapsto Y_M(a,x)u=\sum_{n \in (1/T)\Z}a_{(n)}z^{-n-1}
\] 
such that $x^{r/T}Y_M(a,x)u\in M((x))$ if $a\in V^{(r;g)}$ and $u\in M$. 
The map $Y_M$ satisfies some additional conditions.  
A weak $g$-twisted $V$-module is said to be \textit{$\frac{1}{T}\Zpos$-gradable} 
if $M$ is decomposed into a direct sum of  homogeneous spaces as 
$M=\bigoplus_{n \in (1/T)\Zpos}M(n)$ such that $a_{(m)}M(n)\subset M(d+n-m-1)$ 
for homogeneous $a\in V_d$ and $m,n\in \frac{1}{T}\Z$, where $M(n)=0$ if $n<0$. 
An {\em ordinary} $g$-twisted $V$-module $M$, or a $g$-twisted $V$-module, 
is a weak $g$-twisted $V$-module which is decomposed as  
\[
M=\bigoplus_{\lambda\in \C}M_\lambda,\quad  M_{\lambda}=\{u\in M\mid L_0 u=\lambda u\} 
\]
such that $\dim M_\lambda<\infty$ for any $\lambda\in\C$ and that $M_{\lambda}=0$ 
if the real part of $\lambda$  is sufficiently small. 
If a $g$-twisted $V$-module $M$ is irreducible, then there exists $\rho(M)\in \C$ such that  
\[
M=\bigoplus_{n\in (1/T)\Zpos}M_{\rho(M)+n},\quad M_{\rho(M)}\neq 0. 
\]
We call the number $\rho(M)$ the {\it lowest conformal weight} of $M$. 

Let $M$ be a weak $g$-twisted $V$-module and $h$ an automorphism of $V$. 
Then we have a weak $hgh^{-1}$-twisted $V$-module $(M\circ h,Y_{M\circ h})$ defined by 
$M\circ h=M$ as vector spaces and $Y_{M\circ h}(a,x)=Y_M(h(a),x)$ for $a\in V$.
For a $g$-twisted $V$-module $M$, the restricted dual 
$M'=\bigoplus_{\lambda\in \C}M_{\lambda}^*$ has a structure of a $g^{-1}$-twisted module.
This module is called a {\em contragredient} module of $M$.
If the contragredient module of $V$ is isomorphic to $V$ itself, $V$ is called {\em self-dual}.     

A VOA $V$ is called {\em $g$-rational} ({\em rational} in the case $g=1$ and $T = 1$) 
if any $\frac{1}{T}\Zpos$-gradable $g$-twisted $V$-module is a direct sum of irreducible 
$\frac{1}{T}\Zpos$-gradable $g$-twisted $V$-modules. 
If $V$ is $g$-rational, then every irreducible $\frac{1}{T}\Zpos$-gradable $g$-twisted $V$-module 
becomes a $g$-twisted $V$-module. 
In  the case $g=1$ and $T=1$, we call a weak, $\N$-gradable or ordinary $g$-twisted $V$-module 
a weak, $\N$-gradable or ordinary $V$-module, respectively. 

\subsection{Simple currents and quantum dimensions}\label{Sect2.2}
For a triple of $V$-modules $(M,N,L)$, the fusion rule $N_{M,N}^{L}$ of type $\fusion{M}{N}{L}$ 
is the dimension of the vector space consisting of all intertwining operators of type 
$\fusion{M}{N}{L}$ (see \cite{FHL}). 
If $V$ is rational and $C_2$-cofinite, then the fusion rule of type $\fusion{M}{N}{L}$ is finite 
for any irreducible $V$-modules $M$, $N$ and $L$. 
In this case, we have a commutative ring $K(V)$, called a {\em fusion ring}, 
which is a free $\Z$-module with a basis indexed by $\Irr(V)=\{[M]\}$, 
where $\Irr(V)$ denotes the set of all equivalence classes $[M]$ of irreducible $V$-modules $M$. 
The multiplication of $K(V)$ is defined by 
\[
[M]\times[N]=\sum_{[L]\in \Irr(V)}N_{M,N}^{L}[L]  
\]  
for irreducible $V$-modules $M$ and $N$. 
It is known that $K(V)$ is associative and has a unit $[V]$ 
if $V$ is simple, rational, $C_2$-cofinite and of CFT type. 

A {\em fusion product} of $M$ and $N$ is a $V$-module $M\boxtimes N$ such that 
$[M\boxtimes N]=[M]\times[N]$. 
Fusion products of irreducible modules  exist and are unique up to equivalence if $V$ is simple, 
rational, $C_2$-cofinite and of CFT type. 
An irreducible $V$-module $M$ is called a {\it simple current} 
if the multiplication by $[M]$ on $K(V)$ induces a permutation of $\Irr(V)$.   
A simple, rational, $C_2$-cofinite, self-dual VOA $V$ of CFT type is said to have a 
{\it group-like fusion} if every irreducible module is a simple current.
If $V$ has a group-like fusion, then the fusion product defines  an abelian group structure on $\Irr(V)$. 
 
For a $g$-twisted $V$-module, the character $Z_M(\tau)$ is defined by 
\[
Z_M(\tau)=\tr_M q^{L_0-c_V/24}=\sum_{\lambda\in\C}\dim M_\lambda q^{\lambda-c_V/24},
\quad q=e^{2\pi \sqrt{-1}\tau}
\] 
for $\tau \in \mathcal{H}$, where $\mathcal{H}$ is the upper half plane 
and $c_V$ is the central charge of $V$.   
If the limit 
\[
\lim_{y\rightarrow 0^+} \frac{Z_M(\sqrt{-1}y)}{Z_V(\sqrt{-1}y)}
\]
converges, then the limit is called the {\it quantum dimension} of $M$ over $V$ and denoted by 
$\qdim_V M$ \cite{DongJiaoXu13}. 

\begin{theorem}\label{qdimthm}
Let $V$ be a simple, rational, $C_2$-cofinite, self-dual VOA of CFT type.
Let $g$ be an automorphism of $V$ of finite order $T$, 
and assume that every irreducible $g^i$-twisted     
$V$-module has a positive lowest conformal weight for $i=0,\ldots, T-1$ except $V$.  
Then the following assertions hold. 
\begin{enumerate}
\item[(1)] {\rm (\cite{CarnahanMiyamoto16, Miyamoto15})} 
$V^{\langle g\rangle}$ is rational and $C_2$-cofinite.

\item[(2)] {\rm (\cite{DongJiaoXu13})} For any irreducible $g^i$-twisted $V$-module $M$, 
$\qdim_VM$ exists and it is a algebraic real number greater than or equal to $1$. 

\item[(3)] {\rm (\cite{DongJiaoXu13})} Extend the map 
$\Irr(V)\rightarrow \R, [M]\mapsto \qdim_V M$ to a map $f:K(V)\rightarrow \R$ by $\Z$-linearity. 
Then $f$ is a ring homomorphism.

\item[(4)] {\rm (\cite{DongJiaoXu13})}  
A $V$-module $M$ is a simple current if and only if $\qdim_V M=1$. 

\item[(5)] {\rm (\cite{DongJiaoXu13, DongMason97})} 
Let $M$ be an irreducible $g^i$-twisted $V$-module. 
If $M\circ g\not\cong M$, then $M$ is irreducible as a $V^{\langle g\rangle}$-module. 
Moreover, $\qdim_{V^{\langle g\rangle}}M=T\qdim_{V}M$. 

\item[(6)] {\rm (\cite{DongJiaoXu13, DongRenXu15})}
Let $M$ be an irreducible $g^i$-twisted $V$-module.  
If $M\circ g\cong M$, then $M$ is completely reducible as a $V^{\langle g\rangle}$-module.  
Moreover, if the number of irreducible components in $M$ is $s$, 
then every irreducible component of $M$ has a quantum dimension $\frac{T}{s}\qdim_{V} M$. 
\end{enumerate}  
\end{theorem}

\section{Lattice VOAs and quantum dimensions}\label{Sect3}

In this section, we review the constructions of lattice VOAs and their twisted modules. 
We also calculate their quantum dimensions.  

\subsection{Lattice VOAs and their irreducible modules}\label{Sect3.1}

We first recall the construction of the lattice VOA $V_L$ associated with a positive definite even lattice 
$(L, \la \,\cdot\, , \,\cdot\, \ra)$ (see \cite{Dong93}, \cite{FLM}).  
To construct $V_L$, we consider the free bosonic VOA $M(1)$ associated with a finite dimensional 
vector space $\h=\C\otimes_{\Z} L$ with a $\C$-bilinear form extending the inner product of $L$. 
The VOA $M(1)$ is given by the symmetric algebra $S(\h\otimes t^{-1}\C[t^{-1}])$ as its base space 
and the vacuum vector $\1$ is the unit $1$. 
Its Virasoro vector $\w$ is $\frac{1}{2}\sum_{i=1}^{\dim \h}(h_i\otimes t^{-1})^2$, 
where $\{h_i\}$ is an orthonormal basis of $\h$.    

Let $L^\circ=\{\alpha\in \Q\otimes_{\Z} L\,|\,\langle \alpha,L\rangle\subset \Z\}$ 
be the dual lattice of $L$ and take $s\in \Zplus$ such that $s\langle L^\circ,L^\circ\rangle\subset \Z$. 
We consider a central extension 
\begin{align*}
1\rightarrow \langle\kappa_s\rangle\rightarrow \widehat{L^\circ}\rightarrow L^\circ \rightarrow 1 
\end{align*} 
of $L^\circ$ by a cyclic group $\langle\kappa_s\rangle$ of order $s$.  
We denote the canonical projection of $a\in \widehat{L^\circ}$ to $L$ by $\overline{a}$.  
Such a central extension is determined by an alternating $\Z$-bilinear map    
$c:L^\circ\times L^\circ\rightarrow \Z_s$ subject to
\[
aba^{-1}b^{-1}=\kappa_s^{c(\overline{a},\overline{b})}\qquad \text{ for } a,b\in\widehat{L^\circ}.
\]
Then we have an associative algebra $\C\{L^\circ\}=\C[\widehat{L^\circ}]/(\kappa_s-\xi_s)$. 
For any subset $X\subset L^\circ$, set $\widehat{X}=\{a\in\widehat{L^\circ}|\overline{a}\in X\}$ and 
let $\C\{X\}$ be a subspace of $\C\{L^\circ\}$ spanned by the image of $\widehat{X}$. 
We define $V_X=M(1)\otimes \C\{X\}$ and identify $M(1)=V_{\{0\}}$. 
Then  $V_L=M(1)\otimes \C\{L\}$ is a VOA with the vacuum vector $\1$, and the Virasoro vector 
$\omega$ of $V_L$ coincides with that of $M(1)$.  
For any $\lambda\in L^\circ$, $V_{L+\lambda}$ has a $V_L$-module structure.   

\begin{theorem}\label{thm:rep_lattice_VOA}
Let $L$ be a positive definite even lattice. 
Then the following assertions hold. 
\begin{enumerate}
\item[(1)] {\rm (\cite{Dong93, Zhu96})} $V_L$ is a simple, rational, $C_2$-cofinite VOA 
of CFT type.
\item[(2)] {\rm (\cite{Dong93, DongLepowsky93})} $V_L$ has a group-like fusion and 
$\Irr (V_L)=\{[V_{L+\lambda}]|\lambda\in L^\circ\}\cong L^\circ/L$ with 
$V_{L+\lambda}\boxtimes V_{L+\mu}\cong V_{L+\lambda+\mu}$ for $\lambda,\mu\in L^\circ$.
\end{enumerate}
\end{theorem}

\subsection{$\sigma$-twisted $M(1)$-modules}\label{Sect3.2}
We next review a construction of $\sigma$-twisted modules of $M(1)$ 
(see \cite{BakalovKac04}, \cite{DongLepowsky96}, \cite{Lepowsky85} for the details).  
Let $\h$ be a finite dimensional vector space with nondegenerate symmetric 
bilinear form $\langle\,\cdot\,,\cdot\,\rangle$. 
We consider a linear automorphism $\sigma$ of $\h$ of order $p \ge 2$ 
preserving $\langle\,\cdot\,,\cdot\,\rangle$. 
For simplicity, we assume that $\sigma$ is fixed-point-free on $\h$.  
Following \cite[(4.17)]{BakalovKac04}, \cite[Remark 3.1]{DLM2000}, 
we set 
\[
\h^{(i;\sigma)}=\{u\in \h|\sigma(u)=\xi_{p}^{-i}u\}, \quad i=0,1,\ldots,p-1.
\]   
Note that $\h^{(0;\sigma)} = 0$, for $\sigma$ is fixed-point-free on $\h$. 
Define the $\sigma$-twisted affine Lie algebra $\widehat{\h}[\sigma]$ by  
\[
\widehat{\h}[\sigma]=\bigoplus_{i=1}^{p-1}\h^{(i;\sigma)}\otimes t^{i/p}\C[t,t^{-1}] \oplus \C K
\] 
with the commutation relations 
\begin{align*}
[x\otimes t^m,y\otimes t^n]=\langle x,y\rangle m\delta_{m+n,0}K,\quad [K,\widehat{\h}[\sigma]]=0 
\end{align*}
for $x\in \h^{(i;\sigma)}$, $y\in \h^{(j;\sigma)}$, $m\in i/p+\Z$ and $n\in j/p+\Z$. 
Set
\[
r_i=\dim\h^{(i;\sigma)}. 
\]  

We regard $\C$ as a 
$\bigoplus_{i=1}^{p-1}\widehat{\h}^{(i;\sigma)}\otimes t^{i/p}\C[t]\oplus \C K$-module on which 
$K$ acts as $1$ and $\bigoplus_{i=1}^{p-1}\widehat{\h}^{(i;\sigma)}\otimes t^{i/p}\C[t]$ acts as 
$0$. 
We then construct an induced $\widehat{\h}[\sigma]$-module  
\[
M(1)(\sigma) = U(\widehat{\h}[\sigma])\otimes_{U(\bigoplus_{i=1}^{p-1}\widehat{\h}^{(i;\sigma)}\otimes t^{i/p}\C[t]\oplus \C K)}\C. 
\]

We denote  the action of $h\otimes t^n$ on $M(1)(\sigma)$ by $h(n)$ and set
\[
h(z)=\sum_{n\in \Q}h(n) x^{-n-1}.
\] 
Then there is  a unique  $\sigma$-twisted $M(1)$-module structure $(M(1)(\sigma),Y)$ on 
$M(1)(\sigma)$ such that $Y(h(-1)\1,z)=h(z)$ for $h\in \h$. 
By the construction, we have 
\[
{L}_0=\frac{1}{2}\sum_{i=1}^{p-1}\left(\sum_{m\in i/p+\Z}\sum_{s=1}^{r_i}\NO h^{(i)}_s(m) h^{(p-i)}_s(-m) \NO\right)+\rho({M(1)(\sigma)})
\]  
for a basis $\{h_1^{(i)},\ldots, h_{r_i}^{(i)}\}$ of $\h^{(i;\sigma)}$ satisfying   
\[
\langle h^{(i)}_s,h^{(j)}_t\rangle=\delta_{s,t}\delta_{i+j,p}
\] 
for $1\leq i,j\leq  p-1$, $1\leq s\leq r_i$ and $1\leq t\leq r_j$, where 
\begin{align}\label{tw}
\rho({M(1)(\sigma)})=\frac{1}{4p^2}\sum_{i=1}^{p-1}i(p-i)r_i
\end{align}   
is the lowest conformal weight of $M(1)(\sigma)$ and $\NO \cdot\NO$ denotes 
the normal ordered product. 
We note that 
\begin{equation*}
M(1)(\sigma)=S\big( h^{(p-i)}_s(-i/p-n) \big| 1 \le i \le p-1, 1 \le s \le r_i, n\in \Zpos \big)
\end{equation*}    
is a symmetric algebra generated by 
$ h^{(p-i)}_s(-i/p-n)$, $1 \le i \le p-1$, $1 \le s \le r_i$, $n\in \Zpos$    
as a vector space. 
Since the operator $ h^{(p-i)}_s(-i/p-n)$ increases the conformal weights of homogeneous vectors 
by $i/p+n$, we see that the character of $M(1)(\sigma)$ is given by 
\begin{equation}\label{char001}
Z_{M(1)(\sigma)}(\tau)
=\tr_{M(1)(\sigma)}q^{L_0-{c_{M(1)}}/{24}}
=\frac{q^{\rho(M(1)(\sigma))-\dim \h/24}}{\prod_{i=1}^{p-1}\prod_{n=0}^\infty (1-q^{i/p+n})^{r_{p-i}}}.
\end{equation}    

\subsection{$\widehat{\sigma}$-twisted $V_{L}$-modules}\label{Sect3.3}

In this section,   
we recall a construction of $\widehat{\sigma}$-twisted $V_L$-modules 
\cite{DongLepowsky96, Lepowsky85}.  
Let $(L, \langle \,\cdot\,,\,\cdot\,\rangle)$ be a positive definite even lattice. 
We denote by $O(L)$ the group of all isometries of $L$: 
\[
O(L)=\{ g\in GL(\R\otimes_\Z L)\mid g(L) \subset L \text{ and } 
\la x,y\ra =\la gx, gy\ra \text{ for all } x, y\in L\}.
\] 

Let $\sigma\in O(L)$ be of order $p \ge 2$. 
For simplicity, we assume that $\sigma$ is fixed-point-free on $L$. 
Let $s = p$ if $p$ is even and $s = 2p$ if $p$ is odd. 
Moreover, we assume that $\la \sigma^{p/2}(\alpha), \alpha\ra \in \Z$ for $\alpha \in L$ 
if $p$ is even. 
Following \cite[Remark 2.2]{DongLepowsky96}, 
we define two $\sigma$-invariant alternating $\Z$-bilinear maps 
$c$ and $c^{\sigma}$ from $L \times L$ to $\Z_s$ by 
\begin{equation}\label{eq:c_cs}
c(\alpha,\beta) = \frac{s}{2}\langle \alpha, \beta \rangle + s\Z, \qquad
c^\sigma(\alpha,\beta) 
= \frac{s}{p} \sum_{i=1}^{p-1} \langle  i \sigma^i (\alpha), \beta \rangle + s\Z.
\end{equation} 

The radical of $c^\sigma$ in $L$ is defined by 
\[
R_L^\sigma=\{\alpha\in L\,|\,c^\sigma(\alpha,\beta)=0\text{ for }\beta\in L\}. 
\]

We consider two central extensions
\begin{equation*}
1 \rightarrow \langle\kappa_s\rangle \rightarrow \widehat{L} 
\rightarrow L\rightarrow 1, \qquad 
1 \rightarrow \langle\kappa_s\rangle \rightarrow \widehat{L}_\sigma 
\rightarrow L\rightarrow 1
\end{equation*} 
of $L$ by a cyclic group $\langle \kappa_s \rangle$ of order $s$ 
determined by the $\sigma$-invariant alternating $\Z$-bilinear maps $c$ and $c^{\sigma}$,
respectively.  
We also denote by $\overline{a}$ the canonical projection of 
$a \in \widehat{L}$ or $\widehat{L}_\sigma$ to $L$. 

The isometry $\sigma$ of $L$ can be lifted to an automorphism 
$\widehat{\sigma}$ of the group $\widehat{L}$ of order $p$. 
Moreover, $\widehat{\sigma}$ acts on the group $\widehat{L}_\sigma$
as an automorphism (cf. \cite[Remark 2.2]{DongLepowsky96}). 
We have  
$\overline{\widehat{\sigma}(a)}=\sigma(\overline{a})$ for $a \in \widehat{L}$ or 
$\widehat{L}_\sigma$. 
The automorphism $\widehat{\sigma}$ of $\widehat{L}$ induces an automorphism of the VOA $V_L$ 
associated with the lattice $L$. 
We use the same symbol $\widehat{\sigma}$ to denote the automorphism of $V_L$.

Let $T$ be an $\widehat{L}_\sigma$-module satisfying the following conditions:
\begin{align}
\kappa_s&= \xi_s\id_T,\label{cond003}\\ 
at &= \widehat{\sigma}(a) t\quad \text{for }a\in \widehat{L}_\sigma, t\in T. \label{cond004}
\end{align}
Let $M(1)(\sigma)$ be as in Section \ref{Sect3.2} with $\mathfrak{h} = \C \otimes_\Z L$. 
Then the tensor product 
\[
V_L^T=M(1)(\sigma)\otimes T
\] 
has a $\widehat{\sigma}$-twisted $V_L$-module structure $(V_L^T,Y)$.
The $\widehat{\sigma}$-twisted $V_L$-module $V_L^T$ is irreducible if and only if $T$ is irreducible 
as an $\widehat{L}_\sigma$-module, and every irreducible $\widehat{\sigma}$-twisted $V_L$-module 
is isomorphic to $V_L^T$ for some irreducible $\widehat{L}_\sigma$-module $T$ 
satisfying \eqref{cond003} and \eqref{cond004}. 

Irreducible $\widehat{L}_\sigma$-modules satisfying \eqref{cond003} and \eqref{cond004} 
were constructed in \cite[Section 6]{Lepowsky85}.  
Consider a subgroup $\widehat{K}=\{a^{-1}\widehat{\sigma}(a)|a\in \widehat{L}_\sigma\}$.
Then $\widehat{K}$ is a normal subgroup of $\widehat{L}_{\sigma}$ and 
$\langle\kappa_s\rangle\cap \widehat{K}=\{1\}$. 
Let $\widehat{R}_L^\sigma=\{a\in \widehat{L}_\sigma|\overline{a}\in R_L^\sigma\}$.
Then $\widehat{R}_L^\sigma$ is the center of $\widehat{L}_\sigma$ and 
$\langle\kappa_s\rangle\times \widehat{K}$ is a subgroup of $\widehat{R}_L^\sigma$. 
For any irreducible character $\chi:\widehat{R}_L^\sigma\rightarrow \C^\times$ satisfying 
$\chi(\kappa_s)=\xi_s$ and $\chi(\widehat{K})=\{1\}$, take a maximal abelian subgroup 
$\widehat{A}$ of $\widehat{L}_\sigma$ containing $\widehat{R}_L^\sigma$ and an irreducible 
character $\psi:\widehat{A}\rightarrow \C^\times$ extending $\chi$.
Set 
\[
T_\chi=\C[\widehat{L}_\sigma]\otimes_{\C[\widehat{A}]}\C_\psi,
\]
where $\C_\psi$ is a one dimensional $\widehat{A}$-module affording the character $\psi$. 
Then $T_{\chi}$ is an irreducible $\widehat{L}_\sigma$-module satisfying \eqref{cond003} 
and \eqref{cond004}. 
Note that a different choice of $\psi$ extending $\chi$ gives an irreducible 
$\widehat{L}_\sigma$-module equivalent to $T_{\chi}$   
and there is a one to one correspondence between inequivalent irreducible 
$\widehat{L}_\sigma$-modules satisfying \eqref{cond003} and \eqref{cond004} and 
characters $\chi$ on $\widehat{R}_L^\sigma$ satisfying $\chi(\kappa_s)=\xi_s$ and 
$\chi(\widehat{K})=\{1\}$. 

In particular, there are exactly $|R_L^\sigma/(1-\sigma)(L)|$ inequivalent irreducible 
$\widehat{L}_\sigma$-modules satisfying \eqref{cond003} and \eqref{cond004}.  
Since $\dim T_{\chi} = \dim \C[\widehat{L}_\sigma/\widehat{A}]$  
and $\dim \C[\widehat{L}_\sigma/\widehat{A}]=\dim\C[\widehat{A}/\widehat{R}_L^\sigma]$, 
we have  
\begin{equation}\label{dim002}
(\dim T_{\chi})^2=[\widehat{L}_\sigma:\widehat{R}_L^\sigma]=[L:{R}_L^\sigma]. 
\end{equation} 
 
The following lemma will be used in Section \ref{Sect5.1}.  

\begin{lemma}\label{lem:radical_general}  
Let $L$, $\sigma$, $c^\sigma$ and $R_L^\sigma$ be as above. Then we have
$R_L^\sigma = ((1-\sigma)L^\circ) \cap L$. 
\end{lemma}

\begin{proof}
The isometry $\sigma$ can be extended linearly to a fixed-point-free isometry of $L^\circ$ of order $p$. 
We have
\begin{equation}\label{eq:1-sigma}
(1-\sigma)\sum_{i=1}^{p-1} i \sigma^i 
= \sigma + \sigma^2 + \cdots + \sigma^{p-1} - (p-1)\sigma^p 
= -p
\end{equation}
as operators on $L^\circ$. 
Let $\alpha\in L$. 
Since $c^\sigma(\alpha, \beta) = 0$ is equivqlent to 
$\langle \frac{1}{p} \sum_{i=1}^{p-1} i \sigma^i (\alpha), \beta \rangle \in \Z$ for any $\beta\in L$, 
we have $\alpha \in R_L^\sigma$ if and only if 
$\frac{1}{p} \sum_{i=1}^{p-1} i \sigma^i (\alpha) \in L^\circ$. 
Hence if $\alpha \in R_L^\sigma$, then $-\alpha \in (1-\sigma)L^\circ$ by \eqref{eq:1-sigma}, 
and so $R_L^\sigma \subset ((1-\sigma)L^\circ) \cap L$. 
On the other hand, it follows from \eqref{eq:1-sigma} that 
$\sum_{i=1}^{p-1} i \sigma^i (1-\sigma)(\lambda) = -p\lambda$ for $\lambda \in L^\circ$. 
Thus $R_L^\sigma \supset ((1-\sigma)L^\circ) \cap L$ and the assertion holds.
\end{proof}

\subsection{Quantum dimensions of $\widehat{\sigma}$-twisted modules of $V_L$}\label{Sect3.4}

In this section,  
let $p \ge 2$ be an integer and $\sigma$ an isometry of a positive definite even lattice $L$ 
of order $p$. 
Let $\h=\C\otimes_{\Z} L$.  
Recall the numbers $r_i=\dim \h^{(i;\sigma)}$ for $i=0,1,\ldots, p-1$. 
Following \cite[page 331]{FLM},  
we first prove the following lemma.
The isometry $\sigma$ is not necessarily fixed-point-free for the lemma.  

\begin{lemma}\label{lemma10004}
For any $0\leq i,j\leq p-1$ and positive divisor $k$ of $p$, 
if $\xi_p^i$ and $\xi_p^j$ are both primitive $k$-th root of unity, then $r_i=r_j$.      
\end{lemma} 

\begin{proof}
By taking a basis of $\h$ consisting of elements in $L$, we see that the characteristic polynomial 
$\det(x-\sigma)$ of $\sigma$ on $\h$ has integer coefficients.
Thus all primitive $k$-th roots of unity occur with the same multiplicity, say $n_k$. 
Namely, we have 
\begin{align}\label{eqn800}
\det(x-\sigma)
= \prod_{i=0}^{p-1}(x-\xi_p^{i})^{r_i}
= \prod_{k|p}\Phi_k(x)^{n_k},
\end{align}  
where $\Phi_k(x)$ is the $k$-th cyclotomic polynomial normalized to be monic. 
If $\xi_p^i$ is a $k$-th root of unity,  then $r_i=n_k$.  
\end{proof}

We continue to use the numbers $n_k$ in \eqref{eqn800} for any divisor $k$ of $p$. 
We note that   
\[
x^k-1=\prod_{d|k}\Phi_d(x).
\]
By Poisson's summation formula, we have 
\[
\Phi_k(x)=\prod_{d|k}(x^d-1)^{\mu(k/d)}
\]
where $\mu(k)\in\Z$ is the M{\"{o}}bius function. 
In particular, for any divisor $d$ of $p$, there exists a unique integer $m_d$ such that  
\begin{equation}\label{eqn746}
\det(x-\sigma)=\prod_{d|p}(x^d-1)^{m_d}
\end{equation} 
by \eqref{eqn800}.  
Comparing the degrees of the polynomials on both sides, we have 
\begin{equation}\label{eqn6000}
\sum_{d|p}d m_d=\ell, 
\end{equation}
where $\ell={\rm rank}\, L$. 
Since $x^d-1=\prod_{j=0}^{d-1}(x-\xi_p^{jp/d})$, \eqref{eqn746} implies that   
\begin{equation*}
\prod_{i=0}^{p-1}(x-\xi_p^{i})^{r_i} 
= \prod_{d|p}\prod_{j=0}^{d-1}(x-\xi_p^{jp/d})^{m_d}.
\end{equation*}  

Since $0\leq \frac{jp}{d}<p$ for $j=0,1,\ldots, d-1$, it follows that  
\begin{equation}\label{dim003}
r_i=\sum_{d|p,p|d i}m_d
\end{equation}
for $i=0,1,\ldots, p-1$. 
In particular, if $\sigma$ is fixed-point-free,  then $r_0=0$ and 
\begin{equation}\label{eqn6001}
\sum_{d|p}m_d=0.  
\end{equation}
In this case, we have 
\begin{equation*}
\det(x-\sigma)
=\prod_{i=1}^{p-1}(x-\xi_p^{i})^{r_i}
=\prod_{d|p}\prod_{j=1}^{d-1}(x-\xi_p^{jp/d})^{m_d}.
\end{equation*}  

\begin{remark}\label{remark1011}
If $p$ is a prime and $\sigma$ is fixed-point-free on $L$, then 
\begin{equation*}
\det(x-\sigma)=\Phi_p(x)^{\ell/(p-1)}=(x-1)^{-\ell/(p-1)}(x^{p}-1)^{\ell/(p-1)}.
\end{equation*}
Hence $m_p=-m_1=\frac{\ell}{p-1}$.   
\end{remark}

Now we consider some functions   
\[
a_{c}(\tau)=\prod_{n=0}^\infty(1-q^{c+n}),\quad q=e^{2\pi \sqrt{-1}\tau}
\]
on the upper half plane $\mathcal{H}$ for $c\in \Q_{>0}$. 
For example,   
\[
a_1(\tau)=\prod_{n=1}^\infty(1-q^{n})=q^{-{1}/{24}}\eta(\tau),
\]
where $\eta(\tau)$ is the Dedekind eta function.   

Suppose $\sigma$ is fixed-point-free on $L$. 
Then by \eqref{dim003} and \eqref{eqn6001},  we have  
\begin{align*}
\prod_{j=1}^{p-1}a_{j/p}(\tau)^{r_j}
&=\prod_{j=1}^{p-1}\prod_{d|p,p|dj}a_{j/p}(\tau)^{m_d}\\    
&=\prod_{d|p}\prod_{j=1}^{d-1}a_{j/d}(\tau)^{m_d}\\
&=\prod_{d|p}\prod_{j=1}^{d-1}\prod_{n=0}^\infty(1-q^{j/d+n})^{m_d}\\
&=\frac{\prod_{d|p}\prod_{n=1}^\infty(1-q^{n/d})^{m_d}}{\prod_{d|p}\prod_{n=1}^\infty (1-q^n)^{m_d}}\\
&=\frac{\prod_{d|p}q^{-{m_d}/{24d}}\eta\left({\tau}/{d}\right)^{m_d}}{a_1(\tau)^{\sum_{d|p}m_d}}\\
&=\prod_{d|p}q^{-{m_d}/{24d}}\eta\left({\tau}/{d}\right)^{m_d}.  
\end{align*}
Consequently, if $\sigma$ is fixed-point-free on $L$, 
then \eqref{char001} implies that   
\begin{align*}
Z_{M(1)(\sigma)}(\tau)
&=q^{\rho(M(1)(\sigma))-{\ell}/{24}}\prod_{d|p}\frac{q^{{m_d}/{24d}}}{\eta(\tau /d)^{m_d}}.
\end{align*}     

We recall the character of $V_{L}$:
\begin{align*}
Z_{V_L}(\tau)=\frac{\Theta_L(\tau)}{\eta(\tau)^\ell},
\end{align*}
where $\Theta_L(\tau)$ is the lattice theta function 
\[
\Theta_L(\tau)=\sum_{\alpha\in L}q^{\frac{\langle \alpha,\alpha\rangle}{2}}.
\]
We also consider the ratio 
\begin{equation}\label{eq:Heisenberg_part}
\frac{Z_{M(1)(\sigma)}(\sqrt{-1}y)}{Z_{V_L}(\sqrt{-1}y)} 
= e^{-2\pi y\left(\rho(M(1)(\sigma))-{\ell}/{24}+\sum_{d|p} {m_d}/{24d}\right)}
\times \frac{\eta(\sqrt{-1}y)^\ell}{\prod_{d|p}\eta(\tau/d)^{m_d}}\frac{1}{\Theta_L(\sqrt{-1} y)}.
\end{equation}

\begin{lemma}\label{lemma8834}
If $\sigma$ is fixed-point-free on $L$, then    
\begin{equation*}
\lim_{y\rightarrow 0^+}\frac{\eta(\sqrt{-1} y)^\ell}{\prod_{d|p}\eta(\tau/d)^{m_d}}
\frac{1}{\Theta_L( \sqrt{-1} y)}
=\frac{v}{\sqrt{\prod_{d|p} d^{m_d}}}, 
\end{equation*}
where $v={\rm vol}(\R^\ell/L)$ is the volume of $L$ in $\R^\ell$.
\end{lemma}

\begin{proof}
It is well known that the theta function satisfies the following $S$-transformation formula 
(cf. \cite{Ebeling13}):
\[
\Theta_L(\sqrt{-1}y)=y^{-\ell/2}v^{-1}\Theta_{L^\circ}(\sqrt{-1}/y). 
\]
Since $\eta(-\tau^{-1})=\sqrt{-\sqrt{-1}\tau}\eta(\tau)$, by setting $\tau=\sqrt{-1}d/y$, we have  
\[
\eta(\sqrt{-1}y/d)
=\sqrt{d/y}\eta(\sqrt{-1}d/y)
=\sqrt{d/y}e^{-2\pi d /y}\prod_{n=1}^{\infty}(1-e^{-2\pi n d/ y}). 
\]
Therefore, we have 
\begin{align*}
&\frac{\eta(\sqrt{-1}y)^\ell}{\prod_{d|p}\eta(\sqrt{-1} y /d)^{m_d}}\frac{1}{\Theta_L(\sqrt{-1}y)}\\
&=\frac{e^{-2\pi\ell/y}v\prod_{n=1}^{\infty}(1-e^{-2\pi n/y})^\ell}{\prod_{d|p}\sqrt{y}^{-m_d}\sqrt{d}^{m_d}e^{-2\pi d m_d/y}\prod_{n=1}^{\infty}(1-e^{-2\pi n d/y})^{m_d}  \Theta_{L^\circ}(\sqrt{-1}/y)}\\
&=\frac{e^{-2\pi\ell/y} v\prod_{n=1}^{\infty}(1-e^{-2\pi n/y})^{\ell}}{\sqrt{y}^{-\sum_{d|p}m_d }e^{-2\pi (\sum_{d|p}d m_d)/y}\sqrt{\prod_{d|p}d^{m_d}}\prod_{n=1}^{\infty}(1-e^{-2\pi n d/y})^{m_d}\Theta_{L^\circ}(\sqrt{-1}/y) }.
\end{align*}    

Since $\sum_{d|p}dm_d=\ell$ and $\sum_{d|p}m_d=0$ 
by \eqref{eqn6000} and \eqref{eqn6001}, we have 
\begin{align*}
&\frac{\eta(\sqrt{-1}y)^\ell}{\prod_{d|p}\eta(\sqrt{-1}y/d)^{m_d}}\frac{1}{\Theta(\sqrt{-1}y)}\\
&= \frac{v\prod_{n=1}^{\infty}(1-e^{-2\pi n/y})^{\ell}}{\sqrt{\prod_{d|p}d^{m_d}}\Theta_{L^\circ}(\sqrt{-1}/y)\prod_{n=1}^{\infty}(1-e^{-2\pi n d/y})^{m_d} }\\
&\rightarrow \frac{v}{\sqrt{\prod_{d|p}d^{m_d}} }
\end{align*}    
as $y\rightarrow 0^+$. 
\end{proof}


\begin{theorem}\label{theoremqdimtw}
Let $L$ be a positive definite even lattice of rank $\ell$ and $p\geq 2 $ an integer.  
Let $\sigma$ be an isometry of $ L$ of order $p$ and assume that $\sigma$ is 
fixed-point-free on  $L$. 
Let $\widehat{\sigma} \in \Aut(V_L)$ be a lift of $\sigma$, 
and $V_L^T$ a $\widehat{\sigma}$-twisted $V_L$-module associated with an 
$\widehat{L}_\sigma$-module $T$ satisfying \eqref{cond003} and \eqref{cond004}.
Then the quantum dimension $\qdim_{V_L}V_L^T$ exists and 
\[
\qdim_{V_L}V_L^T=\frac{v\dim T}{\sqrt{\prod_{d|p}d^{m_d}}},
\]
where $v$ is the volume of $L$ in $\R^\ell$ and $m_d$ are integers given by \eqref{eqn746}.  
\end{theorem}

\begin{proof}
The term $e^{-2\pi y\left(\rho(M(1)(\sigma))-{\ell}/{24}+\sum_{d|p}{m_d}/{24d}\right)}$ 
in \eqref{eq:Heisenberg_part} tends to $1$ as $y \to 0^+$. 
Since $V_L^T = M(1)(\sigma) \otimes T$, 
we have 
\[
\qdim_{V_L} V_L^T
=\lim_{y\rightarrow 0^+}\frac{Z_{M(1)(\sigma)}(\sqrt{-1}y)\dim T}{Z_{V_L}(\sqrt{-1}y)}
=\frac{v\dim T}{\sqrt{\prod_{d|p}d^{m_d}}}
\]
by \eqref{eq:Heisenberg_part} and Lemma \ref{lemma8834}. 
\end{proof}  

As a collorary, we have 

\begin{corollary}\label{cor973}
Let $L$, $p$, $\sigma$ and $T$ be as in Theorem \ref{theoremqdimtw}.
Assume that $p$ is a prime and that $T$ is irreducible. 
Then 
\begin{equation*}
(\qdim_{V_L} V_L^T)^2
=p^{-\ell/{(p-1)}}{|L^\circ/L||L/R_L^\sigma|}
=p^{-\ell/{(p-1)}}{|L^\circ/R_L^\sigma|}.  
\end{equation*}   
\end{corollary}

\begin{proof}
It is well known that $v^2=|L^\circ/L|$ for an integral lattice $L$.
By \eqref{dim002}, $(\dim T)^2=|L/R_L^\sigma|$. 
Since $m_p=\ell/{(p-1)}$ by Remark \ref{remark1011}, the assertion follows from 
Theorem \ref{theoremqdimtw}.  
\end{proof}

\begin{remark}
Under the assumption in Corollary \ref{cor973}, we have $L/(1-\sigma)(L)\cong \Z_p^{\ell/(p-1)}$ 
by \cite[Lemma A.1]{GriessLam11} and we can easily check that 
\[
\sum(\qdim_{V_L} V_L^T)^2=|L^\circ/L|,
\]
where $T$ runs through all inequivalent irreducible $\widehat{L}_\sigma$-modules 
satisfying \eqref{cond003} and \eqref{cond004} on the left hand side.  
The number on the right hand side is the global dimension of $V_L$ (see \cite{DongRenXu15}).
\end{remark}

\section{Extensions of a direct sum of the lattice $\sqrt{2}A_{p-1}$}\label{Sect4}

In this section, we describe some lattices between a direct sum of the lattice 
$\sqrt{2}A_{p-1}$ and its dual lattice by using certain codes, 
where $p \ge 3$ is an odd integer.  

\subsection{Codes associated with the lattice $\sqrt{2}A_{p-1}$}\label{Sect4.1}

First, we review some properties of the lattice $\sqrt{2}A_{p-1}$ 
for an odd integer $p \ge 3$, which will be used to develop a theory of code associated with 
the lattice. 

Let $N=\bigoplus_{i=1}^{p-1}\Z \beta_i$ be a positive definite even lattice 
with a base $\{\beta_i\}$ satisfying   
\begin{equation}\label{eq:bibj}
\langle \beta_i,\beta_j\rangle=
\begin{cases}4&\text{if }i=j,\\-2&\text{if }|i-j|=1,\\0 &\text{otherwise.}
\end{cases}
\end{equation}   
The lattice $N$ is usually written as $\sqrt{2} A_{p-1}$ and realized in the Euclidean space 
$\R^{p}$ by setting $\beta_i=\varepsilon_i-\varepsilon_{i+1}$ for $i=1,\ldots, p-1$, where  
\begin{equation}\label{unit001}
\varepsilon_i=(0, \ldots, 0, \sqrt{2}, 0, \ldots, 0)
\end{equation}
and  $\sqrt{2}$ is in the $i$-th entry.   
We set $\varepsilon_0=\varepsilon_p$, $\beta_0=\varepsilon_p-\varepsilon_1$ 
and sometimes regard the indices of $\beta$ and $\varepsilon$ as elements in $\Z_p$.
It is clear that $\beta_0=-\sum_{i=1}^{p-1}\beta_i$. 
  
Set  
\begin{equation}\label{gamma001}
\gamma
=\frac{1}{p}\sum_{i=1}^{p-1} i\beta_i
=\frac{1}{p}(\varepsilon_1 + \cdots + \varepsilon_p) - \varepsilon_p
\end{equation}  
One has 
\begin{equation}\label{inner002}
\langle \gamma,\beta_i\rangle=0, \ 1 \le i \le p-2, 
\quad \langle \gamma,\beta_{p-1}\rangle=-\langle \gamma,\beta_{0}\rangle=2,
\quad \langle \gamma,\gamma\rangle=\frac{2(p-1)}{p}. 
\end{equation}   
Hence $\gamma\in N^\circ$, and $N+\frac{1}{2}\beta_i$ $(i=1,\ldots, p-1)$ and $N+\gamma$ 
generate a subgroup of $N^\circ/N$ isomorphic to $\Z_2^{p-1}\times\Z_p$. 
On the other hand, 
\begin{equation*}
|N^\circ /N|=\det([\langle \beta_i,\beta_j\rangle]_{1\leq i,j \leq p-1})=2^{p-1}\det A=2^{p-1}p,
\end{equation*}  
where $A$ is the Cartan matrix of type $A_{p-1}$:
\[
A=\begin{pmatrix}
2&-1& & \\-1&2&-1& \\ &\ddots&\ddots&\ddots \\& &-1&2&-1 \\& & &-1&2
\end{pmatrix}. 
\]
Thus $N+\frac{1}{2}\beta_i$ $(i=1,\ldots, p-1)$ and $N+\gamma$   
actually form a set of generators of $N^\circ/N$.  

For $u=(u_1,\ldots, u_{p-1})\in \Z^{p-1}$ and $a\in\Z$, set   
\begin{equation}\label{eq:beta_u_a-1}
\beta_{u,a}=\frac{1}{2}\sum_{i=1}^{p-1}u_i\beta_i+a\gamma.  
\end{equation}  
We also set 
\[
L(\overline{u},\overline{a})=N+\beta_{u,a} 
\]
for $\overline{u}=(\overline{u_1},\ldots,\overline{u_{p-1}})\in\Z_2^{p-1}$ 
with $\overline{u_i}=u_i+2\Z\in \Z_2$ and $\overline{a}=a+p\Z\in \Z_p$. 
Then we have an isomorphism
\begin{equation*}
\Z_2^{p-1}\times \Z_{p} \to N^\circ/N; \quad 
(\overline{u},\overline{a}) \mapsto L(\overline{u},\overline{a})
\end{equation*}
of  $\Z$-modules.  

Observe that for $u,v\in\Z^{p-1}$ and $a,b\in\Z$, 
\begin{equation}\label{inner001}
\langle\beta_{u,a},\beta_{v,b}\rangle
= \frac{1}{2}uAv^t + a v_{p-1} + u_{p-1} b + \frac{2(p-1)}{p}ab   
\end{equation}
by \eqref{eq:bibj}, \eqref{inner002} and \eqref{eq:beta_u_a-1}.  
In particular,
\begin{equation}\label{eq:two_inner_prod_values}
\langle\beta_{u,0},\beta_{v,0}\rangle
= \frac{1}{2}uAv^t, \quad 
\langle\beta_{0,a},\beta_{0,b}\rangle
= ab \langle \gamma, \gamma \rangle = \frac{2(p-1)}{p}ab.      
\end{equation}

We see from \eqref{inner001} that
\begin{align}
\langle\beta_{u,a},\beta_{v,b}\rangle
&\equiv \langle\beta_{u,0},\beta_{v,0}\rangle - \frac{2}{p}ab \pmod{\Z},
\label{eq:inner_prod_mod_int}\\
\langle\beta_{u,a},\beta_{u,a}\rangle
&\equiv \langle\beta_{u,0},\beta_{u,0}\rangle - \frac{2}{p}a^2 \pmod{2\Z}.
\label{eq:inner_prod_mod_2}
\end{align}

Since $uAv^t$ is an even integer if $u$ or $v$ lie in $(2\Z)^{p-1}$, 
we have an inner product on $\Z_2^{p-1}$ defined by
\begin{equation}\label{inner003}
\overline{u}\cdot_{A}\overline{v}
=uAv^t + 2\Z
=2\langle\beta_{u,0},\beta_{v,0}\rangle+2\Z \in\Z_2. 
\end{equation}    
That is, $\overline{u}\cdot_{A}\overline{v}=\overline{u}\overline{A}\overline{v}^t$, where 
\[
\overline{A}=\begin{pmatrix}
\overline{0}&\overline{1}& && & \\
\overline{1}&\overline{0}&\overline{1}&&  \\
 &\ddots&\ddots&\ddots& & \\
 & &\overline{1}&\overline{0}&\overline{1}\\
 & & &\overline{1}&\overline{0}
\end{pmatrix}.
\] 
Since $\det A=p\equiv 1$ (mod  $2\Z$), this inner product is nondegenerate. 

We also consider an inner product on $\Z_{p}$ defined by 
\begin{equation}\label{eq:inner_prod_in_Zp}
\overline{a}\cdot_{-2}\overline{b} = -2ab+p\Z 
= p \langle \beta_{0,a}, \beta_{0,b} \rangle + p\Z \in \Z_p.    
\end{equation}  

Note that $\langle \alpha, \alpha \rangle \in \Z$ for $\alpha \in \frac{1}{2}N$. 
Hence $\langle\beta_{u,0},\beta_{u,0}\rangle \in \Z$ for $u \in \Z^{p-1}$. 
Note also that for $\alpha, \beta \in N^\circ$ with $\alpha - \beta \in N$, we have 
$\langle \alpha, \alpha \rangle \equiv \langle \beta, \beta \rangle \pmod{2\Z}$.
We  introduce a weight function $w$ on $\Z_2^{p-1}\times \Z_p$ by  
\begin{equation}\label{eq:weight_kl}  
w(\overline{u},\overline{a})
= \min\{\langle x,x\rangle|\,x\in L(\overline{u},\overline{a})\} \in \frac{1}{p}\Z
\end{equation}
for $(\overline{u},\overline{a})\in \Z_2^{p-1}\times \Z_p$. 
Then
\begin{equation}\label{eq:w_equiv_mod_2}
w(\overline{u},\overline{a})
\equiv \langle x, x \rangle \pmod{2\Z} \quad \text{for } x \in L(\overline{u},\overline{a}).
\end{equation}


In view of \eqref{eq:w_equiv_mod_2}, we define a quadratic form $q : \Z_2^{p-1} \to \Z_2$ by 
\begin{equation}\label{eq:q_on_k}
q(\overline{u}) = w(\overline{u},\overline{0}) + 2\Z = \langle \beta_{u,0}, \beta_{u,0} \rangle + 2\Z
\end{equation}
for $u \in \Z^{p-1}$. 
Then 
\begin{equation}\label{eq:q_on_k_2}
q(\overline{u}) = \frac{1}{2} u A u^t + 2\Z
\end{equation}
by \eqref{eq:two_inner_prod_values} and 
\begin{equation}\label{eq:q_on_k_inner}
q(\overline{u} + \overline{v}) - q(\overline{u}) - q(\overline{v}) = \overline{u} \cdot_{A} \overline{v} 
\quad \text{for } \overline{u}, \overline{v} \in \Z_2^{p-1}.
\end{equation}

We have a $\Z_2$-code $\Z_2^{p-1}$ with an inner product $\cdot_{A}$ and a weight function 
$w(\,\cdot\,,\overline{0})$ and a $\Z_p$-code $\Z_{p}$ with an inner product $\cdot_{-2}$ and 
a weight function $w(\overline{0},\,\cdot\,)$. 
We denote the codes $\Z_2^{p-1}$ and $\Z_p$ by $k,l$, respectively.

Let $d\in\Zplus$. 
We consider nondegenerate inner products on $k^d$, $l^d$ and $k^d\times l^d$ defined by  
\begin{equation}\label{eq:inner_prod_kdld}  
\begin{split}
\overline{\huto{u}} \cdot \overline{\huto{v}} 
= \sum_{i=1}^{d}\overline{{u}_i}\cdot_A\overline{{v}_i}  \in \Z_{2}, \quad 
\overline{\huto{a}} \cdot \overline{\huto{b}}
= \sum_{i=1}^{d}\overline{{a}_i}\cdot_{-2}\overline{{b}_i}  \in \Z_{p},\\
(\overline{\huto{u}},\overline{\huto{a}})\cdot (\overline{\huto{v}},\overline{\huto{b}})
=\sum_{i=1}^{d}(\overline{{u}_i}\cdot_A\overline{{v}_i},
\overline{{a}_i}\cdot_{-2}\overline{{b}_i})\in \Z_{2}\times\Z_p
\end{split}
\end{equation}
for $\overline{\huto{u}}=(\overline{u_1},\ldots,\overline{u_d})$, 
$\overline{\huto{v}}=(\overline{v_1},\ldots,\overline{v_d})\in k^d$ and 
$\overline{\huto{a}}=(\overline{a_1},\ldots,\overline{a_d})$, 
$\overline{\huto{b}}=(\overline{b_1},\ldots,\overline{b_d})\in l^d$.

We also define a weight function $w$ on $ \in k^d\times l^d$ by 
\begin{equation}\label{eq:weight_kdld}  
w(\overline{\huto{u}},\overline{\huto{a}})=\sum_{i=1}^{d}w(\overline{u_i},\overline{a_i})
\end{equation}
and a quadratic form $q : k^d \to \Z_2$ by
\begin{equation}\label{eq:q_on_kd}
q(\overline{\huto{u}}) = w(\overline{\huto{u}},\overline{\huto{0}}) + 2\Z 
= \langle \beta(\huto{u}, \huto{0}), \beta(\huto{u}, \huto{0}) \rangle + 2\Z
\end{equation}
for $\overline{\huto{u}}\in k^d$ and $\overline{\huto{a}}\in l^d$. 
It follows from \eqref{eq:q_on_k_inner} that
\begin{equation}\label{eq:q_on_kd_inner}
q(\overline{\huto{u}} + \overline{\huto{v}}) - q(\overline{\huto{u}}) - q(\overline{\huto{v}}) 
= \overline{\huto{u}} \cdot \overline{\huto{v}} 
\quad \text{for } \overline{\huto{u}}, \overline{\huto{v}} \in k^d.
\end{equation}



For any subset $\mathcal{E}\subset k^d\times l^d$, we set 
\[
\mathcal{E}^\perp=\{(\overline{\huto{u}},\overline{\huto{a}})\in k^d\times l^d|
(\overline{\huto{u}},\overline{\huto{a}})\cdot\mathcal{E}=\{0\}\}.
\]    
Similarly, we define $\mathcal{C}^\perp$ (resp. $\mathcal{D}^\perp$) for a subset 
$\mathcal{C} \subset k^d$ (resp. $\mathcal{D} \subset l^d$). 
We define self-orthogonality and self-duality in a usual manner.
Since $2$ and $p$ are mutually prime, we have the following lemma. 

\begin{lemma}\label{lem2-3}
For any  $\Z$-submodule $\mathcal{E}\subset k^d\times l^d$, there are  $\Z$-submodules 
$\mathcal{C}\subset k^d$ and $\mathcal{D}\subset l^d$ such that 
$\mathcal{E}=\mathcal{C}\times\mathcal{D}$. 
\end{lemma}

In fact, $\mathcal{C}=p\mathcal{E}$ and $\mathcal{D}=2\mathcal{E}$.  
The following proposition is clear from Lemma \ref{lem2-3} and the definition 
\eqref{eq:inner_prod_kdld} of the inner products.  

\begin{proposition}\label{dual001}
Let $\mathcal{C}\subset k^d$ and $\mathcal{D}\subset l^d$ be  $\Z$-submodules.

$(1)$ $(\mathcal{C}\times\mathcal{D})^\perp=\mathcal{C}^\perp\times \mathcal{D}^\perp$.

$(2)$ $\mathcal{C}\times\mathcal{D}$ is self-dual in $k^d \times l^d$ if and only if 
$\mathcal{C}$ and $\mathcal{D}$ are self-dual in $k^d$ and $l^d$, respectively.  
\end{proposition}

For $\huto{u}=(u_1,\ldots,u_d)\in(\Z^{p-1})^d$ and $\huto{a}=(a_1,\ldots,a_d)\in\Z^d$, we set
\[
\beta({\huto{u},\huto{a}})=(\beta_{u_1,a_1},\ldots, \beta_{u_d,a_d})\in (N^\circ)^d
\]
and 
\[
L{(\overline{\huto{u}},\overline{\huto{a}})}=N^{d}+\beta(\huto{u},\huto{a})\subset (N^{\circ})^d,
\] 
where $(N^\circ)^d$ is the orthogonal sum of $d$ copies of $N^\circ$. 
We then have a sublattice  
\[
L_{\mathcal{E}}=\bigcup_{(\overline{\huto{u}},
\overline{\huto{a}})\in\mathcal{E}}L(\overline{\huto{u}},\overline{\huto{a}})
\]
of $(N^\circ)^d$ containing $N^d$ for a $\Z$-submodule $\mathcal{E} \subset k^d \times l^d$. 
Note that $L_{\mathcal{E}}/N^d\cong \mathcal{E}$ as $\Z$-modules. 

\begin{proposition}\label{dual002}
$(L_{\mathcal{E}})^\circ=L_{\mathcal{E}^\perp}$ for any $\Z$-submodule 
$\mathcal{E}\subset k^d\times l^d$.
\end{proposition}

\begin{proof}
Let $\huto{u}=(u_1,\ldots, u_d)$, $\huto{v}=(v_1,\ldots,v_d)\in (\Z^{p-1})^d$ 
and $\huto{a}=(a_1,\ldots,a_d)$, $\huto{b}=(b_1,\ldots,b_d)\in \Z^d$. 
Then
\begin{equation*}
\langle \beta_{u_i,a_i},\beta_{v_i,b_i}\rangle 
\equiv \langle \beta_{u_i,0},\beta_{v_i,0}\rangle - \frac{2}{p}a_i b_i \pmod{\Z}
\end{equation*}
by \eqref{eq:inner_prod_mod_int}. 
Take $w,c\in \Z$ satisfying $w+2\Z=\overline{\huto{u}}\cdot \overline{\huto{v}}$ and 
$c+p\Z=\overline{\huto{a}}\cdot \overline{\huto{b}}$. 
Then we have 
\begin{equation}\label{eq:lattice_inner_general}
\langle \beta(\huto{u},\huto{a}),\beta(\huto{v},\huto{b})\rangle
=\sum_{i=1}^{d}\langle \beta_{u_i,a_i},\beta_{v_i,b_i}\rangle \equiv \frac{w}{2}+\frac{c}{p} 
\pmod{\Z}. 
\end{equation}  
Thus $\langle \beta(\huto{u},\huto{a}),\beta(\huto{v},\huto{b})\rangle\in \Z$ if and only if 
$w\in 2\Z$ and $c\in p\Z$, which is equivalent to 
$\overline{\huto{u}}\cdot\overline{\huto{v}}=\overline{0}$ and 
$\overline{\huto{a}}\cdot\overline{\huto{b}}=\overline{0}$. 

Since $\langle \beta(\huto{u},\huto{a}),N^d\rangle\subset\Z$, we have 
$\langle L(\overline{\huto{u}},\overline{\huto{a}}),L(\overline{\huto{v}},\overline{\huto{b}})\rangle 
\subset  \langle \beta(\huto{u},\huto{a}),\beta(\huto{v},\huto{b})\rangle+\Z$. 
Therefore, 
$L{(\overline{\huto{u}},\overline{\huto{a}})} \subset (L_{\mathcal{E}})^\circ$ if and only if 
$(\overline{\huto{u}},\overline{\huto{a}})\cdot (\overline{\huto{v}},\overline{\huto{b}})=0$ for any 
$(\overline{\huto{v}},\overline{\huto{b}})\in\mathcal{E}$, that is,  
$(\overline{\huto{u}},\overline{\huto{a}})\in\mathcal{E}^\perp$. 
Consequently, we see that 
$L{(\overline{\huto{u}},\overline{\huto{a}})}\subset (L_{\mathcal{E}})^\circ$ if and only if 
$L{(\overline{\huto{u}},\overline{\huto{a}})}\subset L_{\mathcal{E}^\perp }$. 
This completes the proof. 
\end{proof}

The following corollary is clear from Proposition \ref{dual002}. 

\begin{corollary}\label{corollary004}
Let $\mathcal{E}$ be a $\Z$-submodule of $k^d\times l^d$. 

$(1)$ $L_{\mathcal{E}}$ is integral if and only if $\mathcal{E}$ is self-orthogonal. 

$(2)$ $L_{\mathcal{E}}$ is unimodular if and only if $\mathcal{E}$ is self-dual.  
\end{corollary}

For $\huto{u}=(u_1,\ldots, u_d) \in (\Z^{p-1})^d$ and $\huto{a}=(a_1,\ldots,a_d) \in \Z^d$, 
we have
\begin{equation}\label{eq:w_d_equiv_mod_2}
w(\overline{\huto{u}}, \overline{\huto{a}}) 
\equiv \langle x, x \rangle \pmod{2\Z} \quad 
\text{for } x \in L(\overline{\huto{u}},\overline{\huto{a}})
\end{equation}
by \eqref{eq:w_equiv_mod_2}. We also have
\begin{equation}\label{eq:inner_d_prod_mod_2}
\langle \beta(\huto{u},\huto{a}),\beta(\huto{u},\huto{a}) \rangle
\equiv \sum_{i=1}^d \langle \beta_{u_i,0},\beta_{u_i,0} \rangle 
- \frac{2}{p} \sum_{i=1}^d a_i^2 \pmod{2\Z}
\end{equation}
by \eqref{eq:inner_prod_mod_2}. 
Since $\langle \beta_{u_i,0},\beta_{u_i,0} \rangle \in \Z$ and $p \ge 3$ is an odd integer, 
\eqref{eq:inner_d_prod_mod_2} implies that 
$\langle \beta(\huto{u},\huto{a}),\beta(\huto{u},\huto{a}) \rangle \in 2\Z$ if and only if 
$\sum_{i=1}^d \langle \beta_{u_i,0},\beta_{u_i,0} \rangle \in 2\Z$ and 
$\sum_{i=1}^d a_i^2 \in p\Z$.
Note that 
\begin{equation*}
\langle \beta(\huto{u},\huto{0}),\beta(\huto{u},\huto{0}) \rangle 
= \sum_{i=1}^d \langle \beta_{u_i,0},\beta_{u_i,0} \rangle, \quad
\langle \beta(\huto{0},\huto{a}),\beta(\huto{0},\huto{a}) \rangle 
= -\frac{2}{p} \sum_{i=1}^d a_i^2.
\end{equation*}
Thus $\langle \beta(\huto{u},\huto{a}),\beta(\huto{u},\huto{a}) \rangle \in 2\Z$ if and only if 
$\langle \beta(\huto{u},\huto{0}),\beta(\huto{u},\huto{0}) \rangle \in 2\Z$ and 
$\langle \beta(\huto{0},\huto{a}),\beta(\huto{0},\huto{a}) \rangle \in 2\Z$. 
Note that 
$\langle \beta(\huto{u},\huto{0}),\beta(\huto{u},\huto{0}) \rangle \in 2\Z$ 
if and only if $q(\overline{\huto{u}}) = 0$
and that   
$\langle \beta(\huto{0},\huto{a}),\beta(\huto{0},\huto{a}) \rangle \in 2\Z$ 
if and only if $\overline{\huto{a}} \cdot \overline{\huto{a}} = 0$. 

A $\Z$-submodule $\mathcal{E} \subset k^d \times l^d$ 
(resp. $\mathcal{C} \subset k^d$, $\mathcal{D} \subset l^d$)  
is said to be {\em even} 
if $w(\overline{\huto{u}},\overline{\huto{a}})\in 2\Z$ for any 
$(\overline{\huto{u}}, \overline{\huto{a}}) \in\mathcal{E}$ 
(resp. 
$w(\overline{\huto{u}},\overline{\huto{0}})\in 2\Z$ for any $\overline{\huto{u}} \in \mathcal{C}$, 
$w(\overline{\huto{0}},\overline{\huto{a}})\in 2\Z$ for any $\overline{\huto{a}}\in\mathcal{D}$). 
Note that $\mathcal{E}$ (resp. $\mathcal{C}$, $\mathcal{D}$) is even if and only if 
the corresponding lattice $L_{\mathcal{E}}$ (resp. $L_{\mathcal{C} \times \{0\}}$, 
$L_{\{0\} \times \mathcal{D}}$) 
is even by \eqref{eq:w_d_equiv_mod_2}. 

By the above arguments, we have the following proposition.

\begin{proposition}\label{prop002}
Let $\mathcal{C}$ and $\mathcal{D}$ be  $\Z$-submodules of $k^d$ and $l^d$, respectively. 

$(1)$ $L_{\mathcal{C}\times\mathcal{D}}$ is even if and only if 
both $\mathcal{C}$ and $\mathcal{D}$ are even. 

$(2)$ $\mathcal{C}$ is even if and only if $\mathcal{C}$ is totally isotropic with respect to the 
quadratic form $q$.   

$(3)$ $\mathcal{D}$ is even if and only if $D$ is self-orthogonal.
\end{proposition}

The following proposition also holds. 

\begin{proposition}\label{prop233}
Let $\mathcal{C}$ and $\mathcal{D}$ be self-orthogonal $\Z$-submodules of $k^d$ and $l^d$, 
respectively. 
Then for any $\overline{\huto{u}},\overline{\huto{v}}\in\mathcal{C}$ and 
$\overline{\huto{a}},\overline{\huto{b}}\in \mathcal{D}$, 
\begin{equation*}
w(\overline{\huto{u}}+\overline{\huto{v}},\overline{\huto{a}}+\overline{\huto{b}}) \equiv 
w(\overline{\huto{u}},\overline{\huto{a}})
+ w(\overline{\huto{v}},\overline{\huto{b}}) \pmod{2\Z}. 
\end{equation*}
\end{proposition} 

\begin{proof}
By Propositions \ref{dual001}, \ref{dual002} and our assumption, the lattice 
$L_{\mathcal{C}\times\mathcal{D}}$ is integral. 
Then we see that 
\begin{align*}
w(\overline{\huto{u}}+\overline{\huto{v}},\overline{\huto{a}}+\overline{\huto{b}})
&\equiv \langle\beta(\huto{u}+\huto{v},\huto{a}+\huto{b}),
\beta(\huto{u}+\huto{v},\huto{a}+\huto{b})\rangle\\ 
&\equiv \langle\beta(\huto{u},\huto{0}),\beta(\huto{u},\huto{0})\rangle
+ \langle\beta(\huto{v},\huto{0}),\beta(\huto{v},\huto{0})\rangle\\
&\quad + \langle\beta(\huto{0},\huto{a}),\beta(\huto{0},\huto{a})\rangle
+ \langle\beta(\huto{0},\huto{b}),\beta(\huto{0},\huto{b})\rangle\\
&\equiv \langle\beta(\huto{u},\huto{a}),\beta(\huto{u},\huto{a})\rangle
+ \langle\beta(\huto{v},\huto{b}),\beta(\huto{v},\huto{b})\rangle\\
&\equiv w(\overline{\huto{u}},\overline{\huto{a}})
+w(\overline{\huto{v}},\overline{\huto{b}}) \pmod{2\Z} 
\end{align*}
for $\huto{u},\huto{v}\in(\Z^{p-1})^d$, $\huto{a},\huto{b}\in\Z^d$ 
with $\overline{\huto{u}},\overline{\huto{v}}\in\mathcal{C}$ 
and $\overline{\huto{a}},\overline{\huto{b}}\in \mathcal{D}$. 
\end{proof}

\subsection{The isometry $\sigma$ of $N$}\label{Sect4.2}

Let $p \ge 3$ be an odd integer. 
For $i=1,\ldots, p-1$, let $r_i$ be the isometry of $\R^p$ induced by the transposition of 
$i$-th and $(i+1)$-th entries of every vector in $\R^p$. 
Recall that $\beta_i = \varepsilon_i - \varepsilon_{i+1}$ with $\varepsilon_i$ given in \eqref{unit001}. 
Thus $r_i$ is a reflection 
$r_i(x) = x - \frac{2\langle x, \beta_i \rangle}{\langle \beta_i, \beta_i \rangle}\beta_i$ 
with respect to $\beta_i$ and $r_1,\ldots, r_{p-1}$ generate a symmetric group $\mfS_{p}$ 
of degree $p$. 
Since each $r_i$ preserves $N$, we have a group homomorphism $\mfS_p \rightarrow O(N)$ 
via the restriction.
In fact, the group homomorphism induced by the restriction is injective, for 
\[
\R^p=\langle N\rangle_{\R}\oplus \langle(1,\ldots,1)\rangle_{\R}. 
\] 
Hence $\mfS_p$ is a subgroup of $O(N)$. 

The $-1$-isometry $\theta$ on $\R^p$ also gives an isometry of $N$ and 
$O(N) = \mfS_p \times \langle\theta\rangle$.
Since $O(N^\circ)=O(N)$, $O(N)$ acts naturally on $k\times l\cong N^\circ/N$.  

By \eqref{gamma001} and \eqref{inner002}, we have
\begin{equation*}
r_i(\gamma) =
\begin{cases}
\gamma & \text{if } 1 \le i \le p- 2,\\
\gamma - \beta_{p-1} & \text{if } i = p- 1.
\end{cases}
\end{equation*}
Hence $\mfS_p$ acts on $l$ trivially. 
On the other hand, $\theta$ acts on $l$ as a multiplication by $-1$ and acts on $k$ trivially. 

We consider an isometry  
\[
\sigma=r_1\cdots r_{p-1}=(1\,2\,\cdots\,p)\in \mfS_p
\]
of $\R^p$ of order $p$. 
The restriction of $\sigma$ to $N$ gives a fixed-point-free isometry of $N$. 
In fact, $\sigma$ corresponds to a Coxeter element of the Weyl group of type $A_{p-1}$. 
We have $\sigma(\beta_i)=\beta_{i+1}$ for $i=0,1,\ldots, p-1$ and 
$\sigma(\gamma)=\gamma+\beta_0$.
We extend $\sigma$ to isometries of $N^d$ and $(N^\circ)^d$ by diagonal action 
on each direct summand. 
Then $\sigma$ induces an automorphism, which is denoted by the same symbol $\sigma$, 
of $(N^\circ)^d/N^d = k^d\times l^d$ 
preserving inner products and weights. 
We also note that the action of $\sigma$ on $k^d$ is fixed-point-free and 
the action on $l^d$ is the identity.

If $\mathcal{E}$ is a $\sigma$-invariant $\Z$-submodule of $k^d\times l^d$, 
then $\sigma(L_{\mathcal{E}})=L_{\mathcal{E}}$.
Hence the following lemma holds.  
   
\begin{lemma}\label{lem3-6}  
Let $d\in\Zplus$ and $\mathcal{E}$ a $\Z$-submodule of $k^d\times l^d$.
If $\mathcal{E}$ is $\sigma$-invariant, then $\sigma$ gives an isometry of 
$L_{\mathcal{E}}$ of order $p$. 
\end{lemma}

Since $\mfS_p$ acts on $l$ trivially, a $\Z$-submodule $\mathcal{C}\times \mathcal{D}$ of 
$k^d\times l^d$ is $\sigma$-invariant if and only if  $\mathcal{C}$ is $\sigma$-invariant.    
In particular, we have a chain of $\sigma$-invariant lattices   
\[
L_{\mathcal{C}\times \{0\}}\subset L_{\mathcal{C} \times \mathcal{D}}
\subset L_{\mathcal{C}^\perp \times l^d}
\]
when $\mathcal{C}$ is $\sigma$-invariant and self-orthogonal. 
Furthermore, if $\mathcal{C}$ is self-dual, then 
$(L_{\mathcal{C}\times\{0\}})^\circ=L_{\mathcal{C}\times l^d}$ and 
$L_{\mathcal{C}\times\{0\}}+\sigma(\lambda)=L_{\mathcal{C}\times\{0\}}+\lambda$ 
for $\lambda\in (L_{\mathcal{C}\times\{0\}})^\circ$.  
 
%

\section{A construction of VOAs having group-like fusions}\label{Sect5}

In this section, we calculate the quantum dimensions of all irreducible $\widehat{\sigma}^s$-twisted 
modules of $V_{L_{\mathcal{C}\times \mathcal{D}}}$, $1\leq s\leq p-1$, 
where $p$ is an odd prime, 
$\mathcal{C}$ is a $\sigma$-invariant even $\Z$-submodule of $k^d$ and 
$\mathcal{D}$ is an even $\Z$-submodule of $l^d$. 
We show that the orbifold model 
$V_{L_{\mathcal{C}\times \mathcal{D}}}^{\langle \widehat{\sigma} \rangle}$ 
has a group-like fusion if and only if $\mathcal{C}$ is self-dual. 
For such a $\mathcal{C}$, we also discuss the fusion ring of  
$V_{L_{\mathcal{C}\times \mathcal{D}}}^{\langle \widehat{\sigma} \rangle}$.

\subsection{Quantum dimensions of $\widehat{\sigma}^s$-twisted 
$V_{L_{\mathcal{C}\times\mathcal{D}}}$-modules}\label{Sect5.1}

Let $p$ be an odd prime,  
$d \in \Zplus$, 
and $\sigma = r_1\cdots r_{p-1} \in \mfS_p$ as in Section \ref{Sect4.2}. 
Recall that the isometry $\sigma$ induces an automorphism $\sigma$ 
of $k^d\times l^d$ preserving inner products and weights. The automorphism $\sigma$ is 
fixed-point-free on $k^d$ and trivial on $l^d$. 

For $\huto{u}\in(\Z^{p-1})^d$ and $\huto{a}\in\Z^d$, the isometry $\sigma$ transforms the coset 
$L{(\overline{\huto{u}},\overline{\huto{a}})}=N^{d}+\beta(\huto{u},\huto{a})$ of $N^d$ in 
$(N^\circ)^d$ as
\begin{equation*}
\sigma(L{(\overline{\huto{u}},\overline{\huto{a}})}) = 
L{(\sigma(\overline{\huto{u}}),\overline{\huto{a}})}.
\end{equation*}
Therefore, we have
\begin{equation}\label{eq:1-sigma_act_on_coset}
(1-\sigma^s)(L{(\overline{\huto{u}},\overline{\huto{a}})}) \subset   
L{(\overline{\huto{u}} - \sigma^s(\overline{\huto{u}}),\overline{\huto{0}})}
\end{equation}
for $1 \le s \le p-1$. 

Let $\mathcal{C}$ and $\mathcal{D}$ be even $\Z$-submodules of $k^d$ and $l^d$, 
respectively. 
Then $L_{\mathcal{C}\times \mathcal{D}}$ is a positive definite even lattice of rank $d(p-1)$ 
by Proposition \ref{prop002} and its dual lattice is 
$(L_{\mathcal{C}\times \mathcal{D}})^\circ = L_{\mathcal{C}^\perp\times \mathcal{D}^\perp}$ 
by Propositions \ref{dual001} and \ref{dual002}.  
Note that $\mathcal{C} \subset \mathcal{C}^\perp$ and $\mathcal{D} \subset \mathcal{D}^\perp$.

Assume that $\mathcal{C}$ is $\sigma$-invariant. 
Then the lattice $L_{\mathcal{C}\times \mathcal{D}}$ is $\sigma$-invariant 
and $\sigma$ is fixed-point-free of order $p$ on $L_{\mathcal{C}\times \mathcal{D}}$.  
We fix $1 \le s \le p-1$. 
Since $p$ is a prime, $\sigma^s$ is a fixed-point-free isometry of 
$L_{\mathcal{C}\times \mathcal{D}}$ of order $p$. 
Consider a $\sigma^s$-invariant alternating $\Z$-bilinear map 
$c^{\sigma^s} : 
L_{\mathcal{C}\times \mathcal{D}} \times L_{\mathcal{C}\times \mathcal{D}} \to \Z_{2p}$ 
defined by 
\begin{equation}\label{eq:c-sigma-s}
c^{\sigma^s}(\alpha, \beta) = \sum_{i=1}^{p-1} 2 \langle  i \sigma^{si} (\alpha), \beta \rangle + 2p\Z
\quad \text{for } \alpha, \beta \in L_{\mathcal{C} \times \mathcal{D}}
\end{equation}
(cf. \eqref{eq:c_cs}).   
The radical  
$R_{L_{\mathcal{C}\times \mathcal{D}}}^{\sigma^s}
= \{ \alpha \in L_{\mathcal{C}\times \mathcal{D}} | c^{\sigma^s}(\alpha, \beta) 
= 0 \text{ for } \beta \in L_{\mathcal{C}\times \mathcal{D}}\}$ 
of $c^{\sigma^s}$ is 
\begin{equation}\label{eq:radical_of_c_sigma}
R_{L_{\mathcal{C}\times \mathcal{D}}}^{\sigma^s} 
= \big( (1-\sigma^s)((L_{\mathcal{C}\times \mathcal{D}})^\circ)\big) 
\cap L_{\mathcal{C}\times \mathcal{D}}
\end{equation}
by Lemma \ref{lem:radical_general}.

\begin{lemma}\label{lem:radical_data}  
$(1)$ $R_{L_{\mathcal{C}\times \mathcal{D}}}^{\sigma^s} 
= (1 - \sigma^s) L_{\mathcal{C}\times \mathcal{D}^\perp}$.

$(2)$ $R_{L_{\mathcal{C}\times \mathcal{D}}}^{\sigma^s}
/ (1-\sigma^s)L_{\mathcal{C}\times \mathcal{D}} \cong \mathcal{D}^\perp/\mathcal{D}$ 
as $\Z$-modules. 

$(3)$ $(1-\sigma^s)((L_{\mathcal{C}\times \mathcal{D}})^\circ)
/R_{L_{\mathcal{C}\times \mathcal{D}}}^{\sigma^s} \cong \mathcal{C}^\perp/\mathcal{C}$ 
as $\Z$-modules.
\end{lemma}

\begin{proof}  
Since $(L_{\mathcal{C}\times \mathcal{D}})^\circ$ contains 
$L_{\mathcal{C}\times \mathcal{D}^\perp}$, we have 
$R_{L_{\mathcal{C}\times \mathcal{D}}}^{\sigma^s} \supset 
(1 - \sigma^s) L_{\mathcal{C}\times \mathcal{D}^\perp}$ 
by \eqref{eq:1-sigma_act_on_coset} and \eqref{eq:radical_of_c_sigma}.
On the other hand, since $\sigma^s$ is fixed-point-free on $k^d$, the $\Z$-module 
homomorphism $1 - \sigma^s : k^d \to k^d$ is injective, and in fact it is an isomorphism 
of a vector space over $\Z_2$. 
Let $\overline{\huto{u}} \in \mathcal{C}^\perp$. 
Since $\mathcal{C}$ is $\sigma^s$-invariant, we have 
$(1 - \sigma^s) \overline{\huto{u}} \in \mathcal{C}$ if and only if 
$\overline{\huto{u}} \in \mathcal{C}$. 
Thus 
$R_{L_{\mathcal{C}\times \mathcal{D}}}^{\sigma^s} 
\subset (1 - \sigma^s) L_{\mathcal{C}\times \mathcal{D}^\perp}$. 
Hence we obtain (1).

Since $\sigma^s$ is a fixed-point-free isometry of $(N^\circ)^d$, the $\Z$-module homomorphism 
$1 - \sigma^s : (N^\circ)^d \to (1 - \sigma^s) ((N^\circ)^d)$ is an isomorphism. 
Restricting the isomorphism to $L_{\mathcal{C} \times \mathcal{D}^\perp}$ and 
$L_{\mathcal{C} \times \mathcal{D}}$, respectively, we have
\begin{equation*}
 (1 - \sigma^s)L_{\mathcal{C}^ \times \mathcal{D}^\perp} 
/  (1 - \sigma^s)L_{\mathcal{C} \times \mathcal{D}} 
\cong L_{\mathcal{C} \times \mathcal{D}^\perp} / L_{\mathcal{C} \times \mathcal{D}} 
\cong \mathcal{D}^\perp/\mathcal{D}
\end{equation*}
as $\Z$-modules.  
Similarly, 
\begin{equation*}
 (1 - \sigma^s)L_{\mathcal{C}^\perp \times \mathcal{D}^\perp} 
/  (1 - \sigma^s)L_{\mathcal{C} \times \mathcal{D}^\perp} 
\cong L_{\mathcal{C}^\perp \times \mathcal{D}^\perp} 
/ L_{\mathcal{C} \times \mathcal{D}^\perp} 
\cong \mathcal{C}^\perp / \mathcal{C}
\end{equation*}
as $\Z$-modules. Thus (2) and (3) hold. 
\end{proof}

\begin{remark}   
Since $(1-\sigma)(\gamma) = -\beta_0$ and $(1-\sigma)(\beta_i) = \beta_i - \beta_{i+1}$, we have 
$(1-\sigma)L_{\mathcal{C} \times l^d} = L_{\mathcal{C} \times \{0\}} \supset  N^d$.  
Now, $1 - \sigma^s = (1 + \sigma + \cdots + \sigma^{s-1})(1 - \sigma)$ for $2 \le s \le p-1$. 
Since $p$ is an odd prime, $1 + \sigma + \cdots + \sigma^{s-1}$ maps the standard 
simple roots of $A_{p-1}$ to a set of simple roots, and so 
it maps $N^d$ onto $N^d$. 
Hence $(1-\sigma^s)L_{\mathcal{C} \times l^d} = L_{\mathcal{C} \times \{0\}}$.
However, this is not the case if $p$ is not a prime or $\mathcal{D} \ne \{0\}$. 
\end{remark}

Combining Lemma \ref{lem:radical_data} with Corollary \ref{cor973},   
we obtain the following theorem.

\begin{theorem}\label{thm:qdim_irred_twisted_mod}  
Let $p$ be an odd prime, $d \in \Zplus$  
and $\sigma = r_1\cdots r_{p-1} \in \mfS_p$ as in Section \ref{Sect4.2}. 
Let $\mathcal{C}$ and $\mathcal{D}$ be even $\Z$-submodules of $k^d$ and $l^d$, 
respectively. 
Assume that $\mathcal{C}$ is $\sigma$-invariant. 
For $1 \le s \le p-1$, let 
$c^{\sigma^s} : L_{\mathcal{C}\times \mathcal{D}} 
\times L_{\mathcal{C}\times \mathcal{D}} \to \Z_{2p}$  
be a $\sigma^s$-invariant alternating $\Z$-bilinear map defined by \eqref{eq:c-sigma-s}.   
Let $V_{L_{\mathcal{C}\times \mathcal{D}}}^T$ be an irreducible $\widehat{\sigma}^s$-twisted 
$V_{L_{\mathcal{C}\times \mathcal{D}}}$-module associated with an irreducible module $T$ 
satisfying \eqref{cond003} and \eqref{cond004} 
for a central extension $\widehat{L}_{\mathcal{C}\times \mathcal{D}, \sigma^s}$ 
of $L_{\mathcal{C}\times \mathcal{D}}$ by a cyclic group 
$\langle \kappa_{2p} \rangle$ of order $2p$    
determined by $c^{\sigma^s}$. 
Then the quantum dimension 
$\qdim_{V_{L_{\mathcal{C}\times \mathcal{D}}}}V_{L_{\mathcal{C}\times \mathcal{D}}}^T$ 
exists and 
\begin{equation*}
(\qdim_{V_{L_{\mathcal{C}\times \mathcal{D}}}}V_{L_{\mathcal{C}\times \mathcal{D}}}^T)^2
= |\mathcal{C}^\perp / \mathcal{C}|.
\end{equation*}

In particular, 
$\qdim_{V_{L_{\mathcal{C}\times \mathcal{D}}}}V_{L_{\mathcal{C}\times \mathcal{D}}}^T = 1$ 
if and only if $\mathcal{C} = \mathcal{C}^\perp$, that is, $\mathcal{C}$ is self-dual.
\end{theorem} 

\begin{proof}
Since $\sigma^s$ is a fixed-point-free isometry of the rank $d(p-1)$ lattice 
$L_{\mathcal{C}\times \mathcal{D}}$ of order $p$, 
we have
\begin{equation}\label{eq:coker_1-sigma_s}
| (L_{\mathcal{C}\times \mathcal{D}})^\circ 
/ (1 - \sigma^s)((L_{\mathcal{C}\times \mathcal{D}})^\circ) | = p^d
\end{equation}
by \cite[Lemma A.1]{GriessLam11}. Then 
\begin{align*}
|(L_{\mathcal{C}\times \mathcal{D}})^\circ /R_{L_{\mathcal{C}\times \mathcal{D}}}^{\sigma^s}|
&= | (L_{\mathcal{C}\times \mathcal{D}})^\circ / (1 - \sigma^s)((L_{\mathcal{C}\times \mathcal{D}})^\circ) | 
\cdot 
|(1 - \sigma^s)((L_{\mathcal{C}\times \mathcal{D}})^\circ)/R_{L_{\mathcal{C}\times \mathcal{D}}}^{\sigma^s}|\\
&= p^d |\mathcal{C}^\perp/\mathcal{C}|
\end{align*}
by Lemma \ref{lem:radical_data}. 
Therefore, the assertion holds by Corollary \ref{cor973}.
\end{proof}

Note that the dimension of the top level of $V_{L_{\mathcal{C}\times \mathcal{D}}}^T$ is
\begin{equation*}
\dim T = \sqrt{|L_{\mathcal{C} \times \mathcal{D}} 
/ R_{L_{\mathcal{C} \times \mathcal{D}}}^{\sigma^s} |}
= \sqrt{p^d/| \mathcal{D}^\perp/\mathcal{D} |} = |\mathcal{D}|.
\end{equation*}

Next, we study the quantum dimensions of the irreducible modules of the orbifold model 
$V_{L_{\mathcal{C}\times \mathcal{D}}}^{\langle \widehat{\sigma} \rangle}$ 
for a lift $\widehat{\sigma}$ of the isometry $\sigma$.  

\begin{theorem}\label{simplecurrent}  
Let $p$ be an odd prime, $d \in \Zplus$  
and $\sigma = r_1\cdots r_{p-1} \in \mfS_p$ as in Section \ref{Sect4.2}. 
Let $\mathcal{C}$ and $\mathcal{D}$ be even $\Z$-submodules of $k^d$ and $l^d$, 
respectively. 
Assume that $\mathcal{C}$ is $\sigma$-invariant. 
Let $\widehat{\sigma} \in \Aut (V_{L_{\mathcal{C}\times \mathcal{D}}})$ be a lift of $\sigma$ 
of order $p$. 
Then the following conditions are equivalent.

$(1)$ $V_{L_{\mathcal{C}\times \mathcal{D}}}^{\langle \widehat{\sigma} \rangle}$ 
has a group-like fusion.

$(2)$ $\mathcal{C}=\mathcal{C}^\perp$.
\end{theorem}  

\begin{proof}  
For simplicity, we set 
$W = V_{L_{\mathcal{C}\times \mathcal{D}}}^{\langle \widehat{\sigma} \rangle}$. 
Suppose $\mathcal{C} \ne \mathcal{C}^\perp$ and take $\huto{u}\in (\Z^{p-1})^{d}$ 
such that $\overline{\huto{u}} \in \mathcal{C}^\perp - \mathcal{C}$. 
Then as shown in the proof of Lemma \ref{lem:radical_data} (1), we have 
$(1-\sigma)\overline{\huto{u}} \not\in \mathcal{C}$. 
Hence $\sigma(\beta(\huto{u},\huto{0}))-\beta(\huto{u},\huto{0}) 
\not\in L_{\mathcal{C}\times \mathcal{D}}$ and 
$V_{L_{\mathcal{C}\times \mathcal{D}}+\beta(\huto{u},\huto{0})}\circ \sigma 
\not\cong V_{L_{\mathcal{C}\times \mathcal{D}}+\beta(\huto{u},\huto{0})}$ as 
$V_{L_{\mathcal{C}\times \mathcal{D}}}$-modules. 
Then we see from Theorem \ref{qdimthm} (4), (5) and 
Theorem \ref{thm:rep_lattice_VOA} (2) 
that $V_{L_{\mathcal{C}\times \mathcal{D}}+\beta(\huto{u},\huto{0})}$ 
is an irreducible $W$-module with quantum dimension $p$, and in particular it is not a 
simple current. 
Thus (1) implies (2). 

Suppose $\mathcal{C} = \mathcal{C}^\perp$. 
Then 
$(L_{\mathcal{C}\times \mathcal{D}})^\circ = L_{\mathcal{C}\times \mathcal{D}^\perp}$. 
Since $\sigma$ acts on $l^d$ trivially, every irreducible 
$V_{L_{\mathcal{C}\times\mathcal{D}}}$-module 
is $\widehat{\sigma}$-stable and so it decomposes as a direct sum of $p$ irreduclbe 
$W$-modules. 
Then the quantum dimension of each irreducible component is $1$ by 
Theorem \ref{qdimthm} (6). 

By Theorem \ref{thm:qdim_irred_twisted_mod}, we see that every irreducible 
$\widehat{\sigma}^s$-twisted $V_{L_{\mathcal{C}\times\mathcal{D}}}$-module 
is of quantum dimension $1$. 
Each irreducible $\widehat{\sigma}^s$-twisted 
$V_{L_{\mathcal{C}\times\mathcal{D}}}$-module 
also decomposes as a direct sum of $p$ irreducible $W$-modules. Hence 
the quantum dimension of each irreducible component is $1$ by Theorem \ref{qdimthm} (6). 
Since any irreducible $W$-module is an irreducible component of an irreducible 
$V_{L_{\mathcal{C}\times\mathcal{D}}}$-module or 
an irreducible $\widehat{\sigma}^s$-twisted $V_{L_{\mathcal{C}\times\mathcal{D}}}$-module 
for some $1 \le s \le p-1$, we have shown that every irreducible $W$-module is a simple current. 
Thus (2) implies (1). The proof is complete.
\end{proof}

\subsection{Group structure of 
$\Irr(V_{L_{\mathcal{C} \times \mathcal{D}}}^{\langle \widehat{\sigma} \rangle})$}
\label{Sect5.2} 

In this section, using the results in \cite{EMS}, we determine the group structure of 
$\Irr(V_{L_{\mathcal{C} \times \mathcal{D}}}^{\langle \widehat{\sigma} \rangle})$ 
with respect to the fusion product. 

\begin{theorem}
Let $p$ be an odd prime, $d \in \Zplus$  
and $\sigma = r_1\cdots r_{p-1} \in \mfS_p$ as in Section \ref{Sect4.2}. 
Let $\mathcal{C}$ and $\mathcal{D}$ be even $\Z$-submodules of $k^d$ and $l^d$, 
respectively. 
Assume that $\mathcal{C}$ is $\sigma$-invariant. 
Let $r$ be the dimension of $D$ over $\Z_p$, that is, $|D| = p^r$. 
Let $\widehat{\sigma} \in \Aut (V_{L_{\mathcal{C}\times \mathcal{D}}})$ be a lift of $\sigma$ 
of order $p$. 
Suppose $\mathcal{C}= \mathcal{C}^\perp$ and further $3|d$ in the case $p=3$.  
Then $\Irr(V_{L_{\mathcal{C}\times \mathcal{D}}}^{\langle \widehat{\sigma} \rangle})$ 
forms an elementary abelian group of order $p^{d-2r+2}$ with respect to the fusion product. 
\end{theorem}  

\begin{proof}
For simplicity, we set 
$W = V_{L_{\mathcal{C}\times \mathcal{D}}}^{\langle \widehat{\sigma} \rangle}$. 
Since $\mathcal{C}= \mathcal{C}^\perp$, Theorem \ref{simplecurrent} implies that 
$\Irr (W)$ is an abelian group with respect to the fusion product. 

Since $(L_{\mathcal{C}\times \mathcal{D}})^\circ= L_{\mathcal{C} \times \mathcal{D}^\perp}$, 
there are $|\mathcal{D}^\perp/\mathcal{D}| = p^{d-2r}$ inequivalent irreducible 
$V_{L_{\mathcal{C}\times \mathcal{D}}}$-modules, which are all $\widehat{\sigma}$-stable. 
Hence each of these  irreducible $V_{L_{\mathcal{C}\times \mathcal{D}}}$-modules 
decomposes as a direct sum of $p$ irreducible $W$-modules. 
Moreover, Lemma \ref{lem:radical_data} (2) implies that there are $p^{d-2r}$ inequivalent irreducible 
${\widehat{\sigma}}^s$-twisted $V_{L_{\mathcal{C}\times \mathcal{D}}}$-modules 
for each  $1 \leq s\leq p-1$. 
Each of these irreducible 
${\widehat{\sigma}}^s$-twisted $V_{L_{\mathcal{C}\times \mathcal{D}}}$-modules 
decomposes as a direct sum of $p$ irreducible $W$-modules. 
All of these irreducible $W$-modules are inequivalent to each other. 
Therefore, $|\Irr(W)| = p^{d-2r+2}$. 
All irreducible modules except $W$ have positive lowest conformal weights. 

Now, recall from \cite{EMS} that the set $\Irr(W)$ forms a quadratic space with the quadratic form 
$q$ defined by the lowest conformal weight modulo $\Z$; 
$q(M) = \rho(M) + \Z \in \Q/\Z$. 
Moreover, the associated bilinear form 
\begin{equation*}
b(M,N)=q(M\boxtimes N) -q(M)-q(N)
\end{equation*}
 is nondegenerate. 

Since $\mathcal{C}$ is even, \eqref{eq:w_d_equiv_mod_2} and 
\eqref{eq:inner_d_prod_mod_2} imply that    
the lowest conformal weight of any irreducible  
$V_{L_{\mathcal{C}\times\mathcal{D}}}$-module is in $(1/p)\Z$. 
Moreover, the lowest conformal weight of an irreducible $\widehat{\sigma}^s$-twisted 
$V_{L_{\mathcal{C}\times\mathcal{D}}}$-module is given by \eqref{tw}. 
Since $p$ is a prime and $\sigma$ is fixed-point-free on $\h$, 
we have $r_i = \dim \mathfrak{h}^{(i; \sigma^s)} = d$ for $1\leq i\leq p-1$. 
Note that $d$ is an even integer as $\mathcal{C}$ is self-dual. 
Thus the lowest conformal weight of an irreducible $\widehat{\sigma}^s$-twisted 
$V_{L_{\mathcal{C}\times\mathcal{D}}}$-module is 
\[
\frac{d}{4p^2}\sum_{i=1}^{p-1}{i(p-i)} = 
 \frac{d(p-1)(p+1)}{24p}. 
\]
Since $24|(p^2-1)$ if $p$ is a prime with $p> 3$, the conformal weights of irreducible 
$\widehat{\sigma}^s$-twisted $V_{L_{\mathcal{C}\times\mathcal{D}}}$-modules 
are in $\frac{1}p \Z$.
In the case $p=3$, we have $\frac{d}{9}\in\frac{1}{3}\Z$ as $3|d$ by our hypothesis. 

Then, for any $M, N\in \Irr (W)$, we have 
\[
b(M^{\boxtimes p}, N) \equiv p b(M,N) \equiv 0 \quad \pmod{\Z}.
\]
Hence, $M^{\boxtimes p}\cong W$ and $\Irr (W)$ is an elementary abelian $p$-group. 
The proof is complete.
\end{proof}

\appendix
\section{$2$-cocycles and alternating maps}

In this appendix, we show that the $\sigma$-invariant alternating $\Z$-bilinear maps $c$ and 
$c^\sigma$ defined by \eqref{eq:c_cs} for a fixed-point-free isometry $\sigma$ of order $p$ 
can be extended to a rational lattice when $p$ is odd. 
Let $(L, \langle \,\cdot\,,\,\cdot\,\rangle)$ be a positive definite rational lattice 
and $\sigma$ a fixed-point-free isometry of $L$ of order $p$, where $p \ge 3$ is an odd integer.
Define two $\sigma$-invariant $\Z$-bilinear maps 
$\check{\varepsilon}$, $\check{\varepsilon}^\sigma : L \times L \to \Q$ by
\begin{equation*}
\check{\varepsilon}(\alpha,\beta) 
= \frac{1}{2} \sum_{i=1}^{(p-1)/2} \langle \sigma^{i}(\alpha),\beta\rangle, \qquad
\check{\varepsilon}^\sigma(\alpha,\beta) 
= \frac{1}{p} \sum_{i=1}^{(p-1)/2} \langle i \sigma^{i}(\alpha),\beta\rangle
\end{equation*}
and two $\sigma$-invariant alternating $\Z$-bilinear maps 
$\check{c}$, $\check{c}^\sigma : L \times L \to \Q$ by
\begin{equation*}
\check{c}(\alpha, \beta) 
= \check{\varepsilon}(\alpha,\beta) - \check{\varepsilon}(\beta,\alpha), \qquad
\check{c}^\sigma(\alpha, \beta) 
= \check{\varepsilon}^\sigma(\alpha,\beta) - \check{\varepsilon}^\sigma(\beta,\alpha)
\end{equation*}
for $\alpha, \beta \in L$.

Since $\langle\sigma^i(\beta), \alpha\rangle = \langle \beta,\sigma^{p-i}(\alpha) \rangle 
= \langle \sigma^{p-i}(\alpha), \beta \rangle$ and since
\begin{equation*}
1 + \sigma + \cdots + \sigma^{(p-1)/2} = - (\sigma^{(p+1)/2} + \cdots + \sigma^{p-1})
\end{equation*}
as $\sigma$ is fixed-point-free of order $p$, we have
\begin{equation*}
\check{c}(\alpha, \beta) 
= \frac{1}{2}  \langle \alpha, \beta \rangle 
+ \sum_{i=1}^{(p-1)/2} \langle \sigma^{i}(\alpha),\beta\rangle.
\end{equation*}

Let   
\begin{equation*}
f_p(t) = \sum_{i=1}^{(p-1)/2} i(t^i-t^{p-i})\in\Z[t].
\end{equation*}
Then 
$\check{c}^{\sigma}(\alpha,\beta) = \frac{1}{p}\langle f_p(\sigma)(\alpha),\beta\rangle$. 
Since 
\begin{equation*}
f_p(t)
= \sum_{i=1}^{(p-1)/2} i t^i+\sum_{i=(p+1)/2}^{p-1} (i-p) t^i
= \sum_{i=1}^{p-1} i t^i - p \sum_{i=(p+1)/2}^{p-1} t^i,
\end{equation*}
we have
\begin{equation*}
\check{c}^{\sigma}(\alpha,\beta) 
= \frac{1}{p}\sum_{i=1}^{p-1} \langle i \sigma^i(\alpha), \beta \rangle 
- \sum_{i=(p+1)/2}^{p-1} \langle \sigma^i(\alpha), \beta \rangle.
\end{equation*}

Take $s \in 2\Zplus$ such that $s \langle L, L \rangle \subset 2p\Z$ and define    
$\sigma$-invariant $\Z$-bilinear maps
$\varepsilon$, $\varepsilon^\sigma$, $c$, $c^\sigma : L \times L \to \Z_s$ by
\begin{align*}
\varepsilon(\alpha,\beta) &= s \check{\varepsilon}(\alpha,\beta) + s\Z, & \qquad 
\varepsilon^\sigma(\alpha,\beta) &= s \check{\varepsilon}^\sigma(\alpha,\beta) + s\Z,\\
c(\alpha,\beta) &= s \check{c}(\alpha,\beta) + s\Z, & \qquad 
c^\sigma(\alpha,\beta) &= s \check{c}^\sigma(\alpha,\beta) + s\Z
\end{align*}
for $\alpha, \beta \in L$. 
Then
\begin{equation*}
c(\alpha, \beta) 
= \varepsilon(\alpha,\beta) - \varepsilon(\beta,\alpha), \qquad
c^\sigma(\alpha, \beta) 
= \varepsilon^\sigma(\alpha,\beta) - \varepsilon^\sigma(\beta,\alpha).
\end{equation*}
In particular, $c$ and $c^\sigma$ are alternating. Moreover, $\varepsilon$ and $\varepsilon^\sigma$ 
are $2$-cocycles, for they are $\Z$-bilinear.

If $L$ is even, then $\langle \sigma^i(\alpha), \beta \rangle \in \Z$ for $\alpha, \beta \in L$ 
and so the $\sigma$-invariant 
alternating $\Z$-bilinear maps $c$ and $c^\sigma$ 
agree with the ones given in \cite[Remark 2.2]{DongLepowsky96} 
(see \cite[(4.1)]{Lepowsky85} also). 
Therefore, if we take $L = (N^\circ)^d$  and define a $\sigma^s$-invariant alternationg 
$\Z$-bilinear map on $(N^\circ)^d$ as above, then its restriction to 
$L_{\mathcal{C}\times \mathcal{D}}$ coincides with $c^{\sigma^s}$ defined by \eqref{eq:c-sigma-s}.

Similar $2$-cocycles and $\sigma$-invariant alternating $\Z$-bilinear maps were considered 
in \cite{ChenLam16} and \cite{TanabeYamada13} for the case $L$ is the dual lattice of 
$\sqrt{2}A_2$ and $p = 3$.

\end{document}